\newcommand{\si}[1]{#1}    
\newcommand{\jo}[1]{}      

\si{
	\documentclass[10pt,a4paper,twoside]{article}
	\usepackage[utf8]{inputenc}
	\usepackage{graphicx}
	\usepackage{amsfonts}
	\usepackage{amsthm}
	\usepackage{amsmath}
	\usepackage{amssymb}
	\usepackage{array}
	\usepackage{xcolor}
	\usepackage{algorithm}
	\usepackage{indentfirst}
	\usepackage{algorithmic}
	\usepackage{hyperref}
	\usepackage[left=2cm,right=2cm,top=3cm,bottom=2cm]{geometry}
	
	\tracingstats=0
	
	\newtheorem{theorem}{Theorem}[section]
	\newtheorem{proposition}{Proposition}[section]
	\newtheorem{corollary}{Corollary}[section]
	\newtheorem{lemma}{Lemma}[section]
	\newtheorem{definition}{Definition}[section]
	\newtheorem{example}{Example}[section]
	\newtheorem{remark}{Remark}[section]

	\hypersetup{
	  colorlinks=true,
	  citecolor=blue,
	  linkcolor=blue,
	  filecolor=magenta,      
	  urlcolor=cyan,
	}
}

\jo{
	\documentclass[review,onefignum,onetabnum]{siamart190516}
	
	\usepackage{lipsum}
	\usepackage{amsfonts,amsmath,amssymb}
	\usepackage{graphicx}
	\usepackage{algorithmic}
	
	\newsiamremark{remark}{Remark}
	\newsiamremark{hypothesis}{Hypothesis}
	\crefname{hypothesis}{Hypothesis}{Hypotheses}
	\newsiamthm{claim}{Claim}
	\newsiamthm{example}{Example}
}

\usepackage{tikz}
\usetikzlibrary{shapes, shapes.geometric, arrows, positioning}
\tikzstyle{nodo} = [rectangle, rounded corners, minimum width=1cm, minimum height=0.5cm,text centered, draw=black, fill=green!30]
\tikzstyle{arrow} = [thick,->,>=stealth]

\newcommand{\Y}{\mathbb{S}^m}						 
\renewcommand{\S}{\mathbb{S}}					 

\newcommand{\pcon}[1]{\Pi_{\K}(#1)}				 
\newcommand{\bpcon}[1]{\Pi_{\K}\left(#1\right)} 
\newcommand{\seq}[1]{\{#1^k\}_{k\in \N}}		 
\renewcommand{\bar}{\overline}                   
\newcommand{\bd}{\mathrm{bd}}                    
\newcommand{\I}{\mathbb{I}}                      
\newcommand{\K}{\mathbb{S}^m_+}                     
\newcommand{\lin}{\mathrm{lin}}                  
\newcommand{\Ker}{\mathrm{Ker\hspace{0.03cm}}}   
\renewcommand{\Im}{\mathrm{Im\hspace{0.03cm}}}   
\newcommand{\Diag}{\mathrm{Diag }}               
\renewcommand{\int}{\mathrm{int\hspace{0.03cm}}} 
\newcommand{\rank}{\mathrm{rank }}               
\newcommand{\F}{\mathcal{F}}                     
\newcommand{\norm}[1]{\|#1\|}                    
\newcommand{\enorm}[1]{\|#1\|_2}                 
\newcommand{\projs}{\Pi_{\mathbb{S}^m_+}}        
\newcommand{\R}{\mathbb{R}}  					 
\newcommand{\T}{\top\hspace{-1pt}}               
\newcommand{\N}{\mathbb{N}}                      
\newcommand{\xb}{\bar{x}}                        
\newcommand{\Yb}{\bar{Y}}                        
\newcommand{\sym}{\mathbb{S}^m}                  
\newcommand{\A}{\mathcal{A}}
\newcommand{\spn}{\textnormal{span}}
\renewcommand{\limsup}{\mathrm{Lim \hspace{0.05cm} sup}}
\newcommand{\mand}{\textnormal{ and }}
\newcommand{\sdpwcrcq}{weak-CRCQ}
\newcommand{\sdpwcpld}{weak-CPLD}
\newcommand{\sdpscrcq}{seq-CRCQ}
\newcommand{\sdpscpld}{seq-CPLD}
\newcommand{\nlpcrcq}{CRCQ}
\newcommand{\nlpcpld}{CPLD}
\newcommand{\cb}[1]{\mathcal{E}_r(#1)}

\si{
}

\begin{document}

\si{
	\title{Sequential constant rank constraint qualifications for nonlinear semidefinite programming with applications \footnotetext{The authors received financial support from FAPESP (grants 2017/18308-2, 2017/17840-2, and 2018/24293-0), CNPq (grants 301888/2017-5, 303427/2018-3 and 404656/2018-8), and ANID (FONDECYT grant  1201982 and  Basal Program CMM ANID PIA AFB170001).}}

	\author{
	Roberto Andreani \thanks{Department of Applied Mathematics, University of Campinas, Campinas, SP, Brazil. 
		Email: {\tt andreani@unicamp.br}}
	\and
	Gabriel Haeser \thanks{Department of Applied Mathematics, University of S{\~a}o Paulo, S{\~a}o Paulo, SP, Brazil. 
		Emails: {\tt ghaeser@ime.usp.br, leokoto@ime.usp.br}}	
	\and 
	Leonardo M. Mito \footnotemark[2]
	\and
	H{\'e}ctor Ram{\'i}rez C. \thanks{Departamento de Ingenier{\'i}a Matem{\'a}tica and Centro de Modelamiento Matem{\'a}tico (AFB170001 - CNRS IRL2807), Universidad de Chile, Santiago, Chile.
	Email: {\tt hramirez@dim.uchile.cl}}
	}
}

\jo{
	\headers{Sequential constant rank CQs for NSDP}{R. Andreani, G. Haeser, L. M. Mito, and H. Ram{\'i}rez C.}

\title{Sequential constant rank constraint qualifications for nonlinear semidefinite programming with applications \thanks{Submitted to the editors on June 7, 2021. \funding{The authors received financial support from FAPESP (grants 2017/18308-2, 2017/17840-2, and 2018/24293-0), CNPq (grants 301888/2017-5, 303427/2018-3 and 404656/2018-8), and ANID (FONDECYT grant  1201982 and  Basal Program CMM ANID PIA AFB170001).}}}
	%

\author{
		Roberto Andreani \thanks{Department of Applied Mathematics, University of Campinas, Campinas, SP, Brazil (\email{andreani@unicamp.br}).}
	\and 
		Gabriel Haeser \thanks{Department of Applied Mathematics, University of S{\~a}o Paulo, S{\~a}o Paulo, SP, Brazil (\email{ghaeser@ime.usp.br}), (\email{leokoto@ime.usp.br}).}
	\and 
		Leonardo M. Mito\footnotemark[3]
	\and
		H{\'e}ctor Ram{\'i}rez C. \thanks{Departamento de Ingenier{\'i}a Matem{\'a}tica and Centro de Modelamiento Matem{\'a}tico (AFB170001 - CNRS IRL2807), Universidad de Chile, Santiago, Chile (\email{hramirez@dim.uchile.cl}).}
	}

}


\maketitle

\si{
	\abstract{We present new constraint qualification conditions for nonlinear semidefinite programming that extend some of the constant rank-type conditions from nonlinear programming. As an application of these conditions, we provide a unified global convergence proof of a class of algorithms to stationary points without assuming neither uniqueness of the Lagrange multiplier nor boundedness of the Lagrange multipliers set. This class of algorithm includes, for instance, general forms of augmented Lagrangian, sequential quadratic programming, and interior point methods. We also compare these new conditions with some of the existing ones, including the nondegeneracy condition, Robinson's constraint qualification, and the metric subregularity constraint qualification.
	
	\
	
	\textbf{Keywords:} Constant rank, Constraint qualifications, Semidefinite programming, Algorithms, Global convergence.}
}

\jo{
	\begin{abstract}
	  We present new constraint qualification conditions for nonlinear semidefinite programming that extend some of the constant rank-type conditions from nonlinear programming. As an application of these conditions, we provide a unified global convergence proof of a class of algorithms to stationary points without assuming neither uniqueness of the Lagrange multiplier nor boundedness of the Lagrange multipliers set. This class of algorithm includes, for instance, general forms of augmented Lagrangian, sequential quadratic programming, and interior point methods. We also compare these new conditions with some of the existing ones, including the nondegeneracy condition, Robinson's constraint qualification, and the metric subregularity constraint qualification.
	\end{abstract}
	\begin{keywords}
	  Constant rank, Constraint qualifications, Semidefinite programming, Algorithms, Global convergence.
	\end{keywords}
	\begin{AMS}
	  90C46, 90C30, 90C26, 90C22 
	\end{AMS}
}

%
%
%
%
%
%
%
%
%
%
%
%

\section{Introduction}\label{sec:intro}

 \textit{Constraint qualification} (CQ) conditions play a crucial role in optimization. They permit to establish first- and second-order necessary  optimality conditions for local minima and support the convergence theory of many practical algorithms (see, for instance, a unified convergence analysis for a whole class of algorithms by Andreani et al.~\cite[Thm. 6]{ahss12}). Some of the well-known CQs in nonlinear programming (NLP) are the \emph{constant-rank constraint qualification} (\nlpcrcq{}), introduced by Janin~\cite{janin}, and the \emph{constant positive linear dependence} (\nlpcpld{}) condition. The latter was first conceptualized by Qi and Wei~\cite{qiwei}, and then proved to be a constraint qualification by Andreani et al.~\cite{Andreani2005}. Moreover, it has been a source of inspiration for other authors to define even weaker constraint qualifications for NLP, such as the \emph{constant rank of the subspace component} (CRSC)~\cite{CPG}, and the relaxed versions of \nlpcrcq{}~\cite{rcr} and \nlpcpld{}~\cite{ahss12}. Our interest in constant rank-type conditions is motivated, mainly, by their applications towards obtaining global convergence results of iterative algorithms to stationary points without relying on boundedness or uniqueness of Lagrange multipliers. However, several other applications that we do not pursue in this paper may be expected to be extended to the conic context, such as the computation of the derivative of the value function~\cite{janin,param} and the validity of strong second-order necessary optimality conditions that do not rely on the whole set of Lagrange multipliers~\cite{aes2}. Besides, their ability of dealing with redundant constraints, up to some extent, gives modellers some degree of freedom without losing regularity or convergence guarantees on algorithms. For instance, the standard NLP trick of replacing one nondegenerate equality constraint by two inequalities of opposite sign does not violate \nlpcrcq{}, while violating the standard \textit{Mangasarian-Fromovitz CQ} (MFCQ).

Constant-rank type CQs have been proposed in conic programming only very recently. The first extension of \nlpcrcq{} to \textit{nonlinear second-order cone programming} (NSOCP) appeared in \cite{ZZ}, but it was shown to be incorrect in \cite{ZZerratum}. A second proposal~\cite{crcq-naive}, which encompasses also \textit{nonlinear semidefinite programming} (NSDP) problems, consists of transforming some of the conic constraints into NLP constraints via a reduction function, whenever it was possible, and then demanding constant linear dependence of the reduced constraints, locally. This was considered by the authors a {\em naive} extension, since it basically avoids the main difficulties that are expected from a conic framework. 
What both these works have in common is that they somehow neglected the conic structure of the problem.

In a recent article~\cite{weaksparsecq}, we introduced weak notions of regularity for \textit{nonlinear semidefinite programming} (NSDP) that were defined in terms of the eigenvectors of the constraints -- therein called \emph{weak-nondegeneracy} and \emph{weak-Robinson's CQ}. These conditions take into consideration only the diagonal entries of some particular transformation of the matrix constraint. Noteworthy, weak-nondegeneracy happens to be equivalent to the \textit{linear independence CQ} (LICQ) when an NLP constraint is modeled as a structurally diagonal matrix constraint, unlike the standard \textit{nondegeneracy} condition \cite{shapfan}, which in turn is considered the usual extension of LICQ to NSDP. Moreover, the proof technique we employed in~\cite{weaksparsecq} induces a direct application in the convergence theory of an external penalty method. In this paper, we use these conditions to derive our extension proposals for \nlpcrcq{} and \nlpcpld{} to NSDP, which also recover their counterparts in NLP when it is modelled as a structurally diagonal matrix constraint. These CQs are called, in this paper, as \sdpwcrcq{} and \sdpwcpld{}, respectively.

However, to provide support for algorithms other than the external penalty \jo{\linebreak} method, we present stronger variants of these conditions, called sequential-CRCQ and sequential-CPLD (abbreviated \sdpscrcq{} and \sdpscpld{}, respectively), by incorporating perturbations in their definitions. This makes them robust and easily connectible with algorithms that keep track of approximate Lagrange multipliers, but also more exigent. Nevertheless, \sdpscrcq{} is still strictly weaker than nondegeneracy, and independent of Robinson's CQ, while \sdpscpld{} is strictly weaker than Robinson's CQ. On the other hand, \sdpwcrcq{} is strictly weaker than \sdpscrcq{}, while \sdpwcpld{} is strictly weaker than \sdpwcrcq{} and \sdpscpld{}. Moreover, we show that \sdpscpld{} implies the \textit{metric subregularity CQ}.

The content of this paper is organized as follows: Section~\ref{sec:common} introduces notation and some well-known theorems and definitions that will be useful in the sequel. 
Our main results for NSDP are presented in Sections~\ref{sec:weak} and~\ref{sec:seq}. Indeed, Section~\ref{sec:weak} is devoted to the study of \sdpwcrcq{} and \sdpwcpld{} and their properties, which in turn need to invoke weak-nondegeneracy and weak-Robinson's CQ as a motivation.
Section~\ref{sec:seq} studies \sdpscrcq{} and \sdpscpld{} -- the main CQs of this paper -- and some algorithms supported by them. In  Section~\ref{sec:msr}, we discuss the relationship between \sdpscpld{} and the metric subregularity CQ. Lastly, some final remarks are given in Section~\ref{sec:conclusion}.

\if{
{\bf OLD TEXT: WHY SECTION 3 contains the main results and SECTION 4 THE MAIN CQ? I doesn't fit to me ...}
In Section~\ref{sec:weak} we present our main results for NSDP in a constructive reasoning, starting from weak-nondegeneracy and weak-Robinson's CQ, passing through \sdpwcrcq{} and \sdpwcpld{}, to finally arrive at Section~\ref{sec:seq}, where we introduce \sdpscrcq{} and \sdpscpld{} -- the main CQs of this paper -- and some algorithms supported by them. In  Section~\ref{sec:msr}, we discuss the relationship between \sdpscpld{} and the metric subregularity CQ. Lastly, some final remarks are given in Section~\ref{sec:conclusion}.
}\fi

\section{A nonlinear semidefinite programming review}\label{sec:common}

In this section, $\sym$ denotes the linear space of all $m\times m$ real symmetric matrices equipped with the inner product defined as $\langle M,N \rangle\doteq \textnormal{trace}(MN)=\sum_{i,j=1}^m M_{ij}N_{ij}$ for all $M,N\in \sym$, and $\sym_+$ is the cone of all positive semidefinite matrices in $\sym$. Additionally, for every $M\in \Y$ and every $\tau>0$, we denote by $B(M,\tau)\doteq\{Z\in \Y\colon \|M-Z\|<\tau\}$ the open ball centered at $M$ with radius $\tau$ with respect to the Frobenius norm $\|M\|\doteq\sqrt{\langle M,M\rangle}$, and its closure will be denoted by $\bar B(M,\tau)$.

We consider the NSDP problem in standard (dual) form:
\begin{equation}\label{NSDP}
  \tag{NSDP}
  \begin{aligned}
    & \underset{x \in \mathbb{R}^{n}}{\text{Minimize}}
    & & f(x), \\
    & \text{subject to}
    & & G(x) \succeq 0, \\
  \end{aligned}
\end{equation}
where $f\colon\R^{n}\to\R$ and $G \colon\R^n \to \Y $ are continuously differentiable functions, and $\succeq$ is the partial order induced by $\sym_+$; that is, $M\succeq N$ if, and only if, $M-N\in \sym_+$.

Equality constraints are omitted in~\eqref{NSDP} for simplicity of notation, but our definitions and results are flexible regarding inclusion of such constraints, which should be done in the same way as in~\cite{crcq-naive}. Moreover, throughout the whole paper, we will denote the feasible set of~\eqref{NSDP} by $\F$.

Let us recall that the orthogonal projection of an element $M\in \Y$ onto $\K$, which is defined as
\[
	\Pi_{\K}(M)\doteq \underset{N\in \K}{\textnormal{argmin}} \norm{M-N},
\]
is a convex continuous function of $M$ since $\K$ is nonempty, closed, and convex. Furthermore, since $\K$ is self-dual, every $M\in \Y$ has a \textit{Moreau decomposition}~\cite[Prop. 1]{Moreau} in the form
\[
	M=\pcon{M} - \Pi_{\K}(-M)
\]
with $\langle \pcon{M},\Pi_{\K}(-M) \rangle=0$, and a \textit{spectral decomposition} in the form
\begin{equation}\label{sdp:spec1}
	M=\lambda_1(M)u_1(M) u_1(M)^\T+\ldots+\lambda_m(M)u_m(M) u_m(M)^\T,
\end{equation}
where $u_1(M),\ldots,u_m(M)\in \R^m$ are arbitrarily chosen orthonormal eigenvectors associated with the eigenvalues $\lambda_1(M),\ldots,\lambda_m(M)$, respectively. In turn, these eigenvalues are assumed to be arranged in non-increasing order. Equivalently, we can write~\eqref{sdp:spec1} as $M=U \mathcal{D} U^\T$, where $U$ is an orthogonal matrix whose $i$-th column is $u_i(M)$, and $\mathcal{D}\doteq\textnormal{Diag}(\lambda_1(M),\ldots,\lambda_m(M))$ is a matrix whose diagonal entries are $\lambda_1(M),\ldots,\lambda_m(M)$ and the remaining entries are zero.

A convenient property of the orthogonal projection onto $\sym_+$ is that, for every $M\in \sym$, we have
\[
	\projs(M)= [\lambda_1(M)]_+u_1(M) u_1(M)^\T+\ldots+[\lambda_m(M)]_+u_m(M) u_m(M)^\T,
\]
where $[ \ \cdot \ ]_+\doteq\max\{ \ \cdot \ , 0\}$.

Given a sequence of sets $\seq{S}$, recall its \textit{outer limit} (or \textit{upper limit}) in the sense of Painlev{\'e}-Kuratowski (cf.~\cite[Def. 4.1]{rwets} or ~\cite[Def. 2.52]{bshapiro}), defined as
\[
	\limsup_{k\in \N} S^k\doteq \left\{y\colon \exists I\subseteq_{\infty} \N, \ \exists \{y^k\}_{k\in I}\to y, \ \forall k\in I, \ y^k\in S^k \right\},
\]
which is the collection of all cluster points of sequences $\seq{y}$ such that $y^k\in S^k$ for every $k\in \N$. The notation $I\subseteq_{\infty} \N$ means that $I$ is an infinite subset of the set of natural numbers $\N$.

We denote the \emph{Jacobian} of $G$ at a given point $x\in \R^n$ by $DG(x)$, and the \emph{adjoint} operator of $DG(x)$ will be denoted by $DG(x)^*$. Moreover, the $i$-th partial derivative of $G$ at $x$ will be denoted by $D_{x_i} G(x)$, and the \textit{gradient} of $f$ at $x$ will be written as $\nabla f(x)$, for every $x\in \R^n$. 

\subsection{Classical optimality conditions and constraint qualifications}

As usual in continuous optimization, we drive our attention towards local solutions of~\eqref{NSDP} that satisfy the so-called \emph{Karush-Kuhn-Tucker} (KKT) conditions, defined as follows:

\begin{definition}\label{def:KKT}
	We say that the \emph{Karush-Kuhn-Tucker} conditions hold at $\xb\in\F$ when there exists some $\Yb\succeq 0$ such that
	\begin{equation*}\label{eq:KKT}
		\nabla_x L(\xb,\Yb)=0 \quad \mand \quad \langle G(\xb),\Yb\rangle=0,
	\end{equation*}
	where $L(x,Y)\doteq f(x)-\langle G(x),Y\rangle$ is the Lagrangian function of~\eqref{NSDP}. The vector $\Yb$ is called a \emph{Lagrange multiplier} associated with $\xb$, and the set of all Lagrange multipliers associated with $\xb$ will be denoted by $\Lambda(\xb)$.
\end{definition}
Of course, not every local minimizer satisfies KKT in the absence of a CQ. In order to recall some classical CQs, it is necessary to use the \emph{(Bouligand) tangent cone} to $\K$ at a point $M\succeq 0$. This object can be characterized in terms of any matrix $E\in \R^{m\times m-r}$, whose columns form an orthonormal basis of $\Ker M$, as follows (e.g., \cite[Ex. 2.65]{bshapiro}):
\begin{equation}\label{eq:tangent_cone_SDP}
	T_{\S^m_+}(M) = \{N\in \S^m\colon E^\top N E \succeq 0\},
\end{equation}
where $r$ denotes the rank of $M$.
%
%
%
So, its \emph{lineality space}, 
defined as the largest linear space contained in $T_{\K}(M)$, is computed as follows: 
\begin{equation}\label{eq:lin_tangent_cone_SDP}
	\lin(T_{\S^m_+}(M)) = \{N\in \S^m\colon E^\top N E = 0\}.
\end{equation}
The latter is a direct consequence of the identity $\lin(C)=C\cap (-C)$, satisfied for any closed convex cone $C$.

One of the most recognized constraint qualifications in NSDP is the \emph{nondegeneracy} (or \emph{transversality}) condition introduced by Shapiro and Fan~\cite{shapfan}, which can be characterized~\cite[Eq. 4.172]{bshapiro} at a point $\xb\in \F$ when the following relation is satisfied:
\[
	\Im DG(\xb) + \lin(T_{\K}(G(\xb)))=\Y.
\]
If $\xb$ is a local solution of~\eqref{NSDP} that satisfies nondegeneracy, then $\Lambda(\xb)$ is a singleton, but the converse is not necessarily true unless $G(\xb)+\Yb\in \int\K$  holds for some $\Yb\in \Lambda(\xb)$~\cite[Prop. 4.75]{bshapiro}. This last condition is known as \textit{strict complementarity} in this NSDP framework. Here, $\int\K$ stands for the topological interior of $\K$. By~\eqref{eq:tangent_cone_SDP} it is possible to characterize nondegeneracy at $\xb$ by means of any given matrix $\bar{E}$ with orthonormal columns that span $\Ker G(\xb)$. Indeed, following~\cite[Sec. 4.6.1]{bshapiro}, nondegeneracy holds at $\xb$ if, and only if, either $\Ker G(\xb)=\{0\}$ or the linear mapping $\psi_{\xb}\colon\R^n \to \S^{m-r}$ given by
\begin{equation}
\label{def:psi_trans}
\psi_{\xb}( \ \cdot \ ) \doteq \bar{E}^\top  D  G(\xb) [ \ \cdot \ ] \bar{E}
\end{equation}
is surjective, which is in turn equivalent to saying that the vectors
\begin{equation}\label{def:vii}
	\begin{aligned}
		v_{ij}(\xb,\bar{E}) & \doteq \left[\bar{e}_i^\T D_{x_1}G(\xb) \bar{e}_j, \ldots, \bar{e}_i^\T D_{x_n}G(\xb) \bar{e}_j\right]^\T, \quad 1\leq i\leq j\leq m-r,\\
	\end{aligned}
\end{equation}
are linearly independent~\cite[Prop. 6]{Shapiro1997}, where $\bar{e}_i$ denotes the $i$-th column of $\bar{E}$ and $r$ is the rank of $G(\xb)$. 

Another widespread constraint qualification is Robinson's CQ~\cite{Robinson1976}, which can be characterized at $\xb\in \F$ by the existence of some $d\in \R^n$ such that 
\begin{equation}\label{def:RobCQ}
	G(\xb)+DG(\xb)[d]\in \int\K.
\end{equation}
It is known (e.g., \cite[Props. 3.9 and 3.17]{bshapiro}) that when $\xb$ is a local solution of~\eqref{NSDP}, then $\Lambda(\xb)$ is nonempty and compact if, and only if, Robinson's CQ holds at $\xb$. 

Given the properties and characterizations recalled above, the nondegeneracy condition is typically considered the natural extension of LICQ from NLP to NSDP, while Robinson's CQ is considered the extension of MFCQ.

\subsection{A sequential optimality condition connected to the external penalty method}

If we do not assume any CQ, every local minimizer of~\eqref{NSDP} can still be proved to satisfy at least a \textit{sequential} type of optimality condition that is deeply connected to the classical external penalty method. Namely: 
\begin{theorem}\label{thm:minakkt}
	Let $\xb$ be a local minimizer of~\eqref{NSDP}, and let $\{\rho_k\}_{k\in \N}\to +\infty$. Then, there exists some $\seq{x}\to \xb$, such that for each $k\in \N$, $x^k$ is a local minimizer of the regularized penalized function 
	\[
		F(x)\doteq f(x)+\frac{1}{2}\enorm{x-\xb}^2 
	  	+ \frac{\rho_k}{2} \norm{\Pi_{\K}(-G(x))}^2.
	\]
\end{theorem}
\begin{proof}
See~\cite[Thm. 3.2]{ahv}. For a more general proof, see the first part of the proof of~\cite[Thm. 2]{Andreani2020}.
\end{proof}

Note that Theorem~\ref{thm:minakkt} provides a sequence $\seq{x}\to \xb$ such that each $x^k$ satisfies, with an error $\varepsilon^k\to 0^+$, the first-order optimality condition of the unconstrained minimization problem
\[
	\underset{x\in \R^n}{\textnormal{Minimize}} \ f(x)+\frac{\rho_k}{2}\|\Pi_{\K}(-G(x))\|^2,
\]
so $\seq{x}$ characterizes an output sequence of an~\emph{external penalty method}. Moreover, the sequence \si{\linebreak} $\seq{Y}\subseteq \K$, where 
\[
	Y^k\doteq \rho_k \Pi_{\K}(-G(x^k))
\]
for every $k\in \N$, consists of \emph{approximate Lagrange multipliers} for $\xb$, in the sense that $\nabla_x L (x^k,Y^k)\to 0$ and complementarity and feasibility are approximately fulfilled, in view of Moreau's decomposition -- indeed, note that $\langle G(x^k)+\Delta^k,Y^k\rangle=0$ and $G(x^k)+\Delta^k\succeq 0$, with $\Delta^k=-\Pi_{\K}(-G(x^k))\to 0$, for every $k\in \N$.

These sequences will suffice to obtain the results of the first part of this paper (Section~\ref{sec:weak}), but in order to extend their scope to a larger class of iterative algorithms, in Section~\ref{sec:seq}, we will need a more general sequential optimality condition, which will be presented later on. 

\subsection{Reviewing constant rank-type constraint qualifications for NLP}

This section is meant to be a brief review of the main results regarding the classical nonlinear programming problem:
\begin{equation}\label{NLP}
  \tag{NLP}
  \begin{aligned}
    & \underset{x \in \mathbb{R}^{n}}{\text{Minimize}}
    & & f(x), \\
    & \text{subject to}
    & & g_1(x)\geq 0, \ldots, g_m(x)\geq 0, \\
  \end{aligned}
\end{equation}
where $f, g_1,\ldots,g_m\colon \R^n\to \R$ are continuously differentiable functions. 

As far as we know, the first constant rank-type constraint qualification was introduced by Janin~\cite{janin}, to obtain directional derivatives for the optimal value function of a perturbed NLP problem. Janin's condition is defined as follows:

\begin{definition}\label{nlp:crcq}
Let $\xb\in \F$. The \emph{constant rank constraint qualification} \jo{\linebreak} for~\eqref{NLP} (\nlpcrcq{}) holds at $\xb$ if there exists a neighborhood $\mathcal{V}$ of $\xb$ such that, for every subset $J\subseteq \{i\in \{1,\ldots,m\}\colon g_i(\xb)=0\}$, the rank of the family $\{\nabla g_i(x)\}_{i\in J}$ remains constant for all $x\in \mathcal{V}$.
\end{definition}

As noticed by Qi and Wei~\cite{qiwei} it is possible to rephrase Definition~\ref{nlp:crcq} in terms of the ``constant linear dependence'' of $\{\nabla g_i(x)\}_{i\in J}$ for every $J$. That is, \nlpcrcq{} holds at $\xb$ if, and only if, there exists a neighborhood $\mathcal{V}$ of $\xb$ such that, for every $J\subseteq \{i\in \{1,\ldots,m\}\colon g_i(\xb)=0\}$, if $\{\nabla g_i(\xb)\}_{i\in J}$ is linearly dependent, then $\{\nabla g_i(x)\}_{i\in J}$ remains linearly dependent for every $x\in \mathcal{V}$. Based on this characterization, Qi and Wei proposed a relaxation of \nlpcrcq{}, which they called \emph{constant positive linear dependence} (\nlpcpld{}) condition, but this was only proven to be a constraint qualification a few years later, in~\cite{Andreani2005}. To properly define \nlpcpld{}, recall that a family of vectors $\{z_i\}_{i\in J}$ of $\R^n$ is said to be \emph{positively linearly independent} when
\[
\sum_{i\in J} z_i \alpha_i=0, \ \ \alpha_i\geq 0, \ \forall i\in J \quad \Rightarrow \quad \alpha_i=0, \ \forall i\in J.
\]
Next, we recall the CPLD constraint qualification:

\begin{definition}\label{nlp:cpld}
Let $\xb\in \F$. The \emph{constant positive linear dependence} condition for~\eqref{NLP} (\nlpcpld{}) holds at $\xb$ if there exists a neighborhood $\mathcal{V}$ of $\xb$ such that, for every $J\subseteq \{i\in \{1,\ldots,m\}\colon g_i(\xb)=0\}$, if the family $\{\nabla g_i(\xb)\}_{i\in J}$ is positively linearly dependent, then $\{\nabla g_i(x)\}_{i\in J}$ remains linearly dependent for all $x\in \mathcal{V}$.
\end{definition}

Clearly, \nlpcpld{} is implied by \nlpcrcq{}, which is in turn implied by LICQ and is independent of MFCQ. Moreover, \nlpcpld{} is implied by MFCQ, and all those implications are strict~\cite{Andreani2005,janin}. To show that our extensions of \nlpcrcq{} and \nlpcpld{} to NSDP are indeed constraint qualifications (Theorem~\ref{sdp:mincrcqkkt}), we shall take inspiration in~\cite{ahss12}, where the authors employ Theorem~\ref{thm:minakkt} together with the well-known \emph{Carath{\'e}odory's Lemma}:

\begin{lemma}[Exercise B.1.7 of~\cite{bertsekasnl}]\label{lem:carath} Let $z_1,\dots,z_{p}\in\mathbb{R}^n$, and let $\alpha_1,\ldots,\alpha_{p}\in \R$ be arbitrary. Then,
there exists some $J \subseteq  \{1, \ldots, p \}$ and some scalars $\tilde{\alpha}_{i}$ with $i \in J$, such that $ \{z_{i} \}_{i \in J} $ is linearly independent,
\[
	 \sum_{i=1}^{p} \alpha_{i} z_{i} = \sum_{i \in J} \tilde{\alpha}_{i} z_{i},
\]
and $\alpha_{i}\tilde{\alpha}_i>0$, for all $i\in J$.
\end{lemma}

See also~\cite{haeserip}. If one considers equality constraints in~\eqref{NSDP} separately, one should employ an adapted version of Carath{\'e}odory's Lemma that fixes a particular subset of vectors, which can be found in~\cite[Lem. 2]{ahss12}. In our current setting, Lemma~\ref{lem:carath} will suffice as is. 

%
%
%
%
%
%
%
%
%
%
%
%

\section{Constant rank constraint qualifications for NSDP}\label{sec:weak}

Based on the relationship between LICQ and \nlpcrcq{}, the most natural candidate for an extension of \nlpcrcq{} to NSDP is to demand every subset of
\[
	\{v_{ij}(x,\bar{E})\colon 1\leq i\leq j\leq m-r\}
\] to remain with constant rank (or constant linear dependence) in a neighborhood of $\xb$. However, this candidate cannot be a CQ, as shown in the following counterexample, adapted from~\cite[Eq. 2]{ZZerratum}:
\begin{example}\label{ex:sdpnino}
	Consider the problem to minimize $f(x)\doteq -x$ subject to
	\[
		G(x)\doteq 
		\begin{bmatrix}
			x & x+x^2\\
			x+x^2 & x
		\end{bmatrix}
		\succeq 0 .
	\]
	For this problem, $\xb\doteq 0$ is the only feasible point and, therefore, the unique global minimizer of the problem. 
	\if{
	Let us compute $v_{ij}(x,E)$ in terms of an arbitrary orthogonal matrix $E$, whose columns are denoted by $e_1\doteq[a, b]^\T$ and $e_2\doteq[c, d]^\T$, where $a^2+b^2=c^2+d^2=1$, and $ac+bd=0$:
	\begin{equation*}
		\begin{array}{l}
			v_{11}(x,E)=
		1+2ab(1+2x),\\
			v_{22}(x,E)=
		1+2cd(1+2x),\\
			v_{12}(x,E)=
		(ad+bc)(1+2x).
		\end{array}
	\end{equation*}
	}\fi
	Since $G(\xb)=0$, the columns of the matrix $\bar{E}\doteq\I_2$
	form an orthonormal basis of $\Ker G(\xb)$ (the whole space $\R^2$). For this choice of $\bar{E}$, we have 
	\[
		v_{11}(x,\bar{E})=v_{22}(x,\bar{E})=1 \quad \mand \quad v_{12}(x,\bar{E})=1+2x.
	\]
	Since they are all bounded away from zero, the rank of every subset of $\{v_{ij}(x,\bar{E})\colon 1\leq i\leq j\leq 2\}$ remains constant for every $x$ around $\xb$.
	However, Note that $\xb$ does not satisfy the KKT conditions because any $\Yb\doteq \begin{bmatrix}\Yb_{11}&\Yb_{12}\\ \Yb_{12}&\Yb_{22}\end{bmatrix}\in \Lambda(\xb)$ would necessarily be a solution of the system 
	\[
		\begin{array}{l}
			\Yb_{11}\geq 0,\\
			\Yb_{22}\geq 0,\\
			\Yb_{11}\Yb_{22}-\Yb_{12}^2\geq 0,\\
			\Yb_{11}+2\Yb_{12}+\Yb_{22}=-1,
		\end{array}
	\] 
	which has no solution.

\end{example}

Besides, it is well-known that even if $G$ is affine, not all local minimizers of~\eqref{NSDP} satisfy KKT, but in this case every subfamily of $\{v_{ij}(x,\bar{E})\colon 1\leq i\leq j\leq m-r\}$ remains with constant rank for every $x\in \R^n$.

What Example~\ref{ex:sdpnino} tells us is that $\bar{E}=\I_2$ may be a bad choice of $\bar{E}$. In fact, let us choose a different $\bar{E}$, namely, denote the columns of $\bar{E}$ by $\bar{e}_1\doteq[a, b]^\T$ and $\bar{e}_2\doteq[c, d]^\T$, 
and take $a=-1/\sqrt{2}$ and $b=c=d=1/\sqrt{2}$. This election of $\bar{E}$ happens to diagonalize $G(x)$ for every $x$, but it follows that
	\begin{equation*}
		\begin{array}{l}
			v_{11}(x,\bar E)=
		1+2ab(1+2x) =-2x ; \\
			v_{22}(x,\bar E)= 
		1+2cd(1+2x) =2 (1+x);\\
			v_{12}(x,\bar E)=
		(ad+bc)(1+2x)= 0,
		\end{array}
	\end{equation*}
and the rank of $\{v_{11}(x,\bar{E})\}$ does not remain constant in a neighborhood of $\xb=0$.

In light of our previous work \cite{weaksparsecq}, the situation presented above is not surprising. Therein, we already noted that identifying the ``good" matrices $\bar E$ allows us to obtain relaxed versions of nondegeneracy and Robinson's CQ for NSDP.
This identification can also be used to extend constant-rank type conditions  to NSDP and is the starting point for the results we will present in the current manuscript.  

For the sake of completeness, let us quickly summarize a discussion raised in~\cite{weaksparsecq} before presenting the results of this paper. Consider a feasible point $\xb\in \F$ and denote by $r$ the rank of $G(\xb)$. Observe that $\lambda_{r}(M)>\lambda_{r+1}(M)$ for every $M\in \S^m$ close enough to $G(\xb)$. Thus, when $r<m$, define the set
\begin{equation}\label{def:er}
\cb{M}\doteq \left\{  E\in \R^{m\times m-r}\colon \begin{array}{l} ME=E\Diag(\lambda_{r+1}(M),\ldots,\lambda_{m}(M)) \\ E^\T E=\I_{m-r}\end{array} \right\},
\end{equation}
which consists of all matrices whose columns are orthonormal eigenvectors associated with the $m-r$ smallest eigenvalues of $M$, which is well defined whenever $\lambda_{r}(M)>\lambda_{r+1}(M)$. In~\eqref{def:er}, $\Diag(\lambda_{r+1}(M),\ldots,\lambda_{m}(M))$ denotes the diagonal matrix whose diagonal entries are $\lambda_{r+1}(M),\ldots,\lambda_{m}(M)$. By convention, $\mathcal{E}_r(M)\doteq\emptyset$ when $r=m$. By construction, $\cb{M}$ is nonempty provided $r<m$ and $M$ is close enough to $G(\xb)$. In particular, in this situation, $\cb{G(\xb)}$ is the set of all matrices with orthonormal columns that span $\Ker G(\xb)$.

We showed, in~\cite[Prop. 3.2]{weaksparsecq}, that nondegeneracy can be equivalently stated as the linear independence of the smaller family, $\{v_{ii}(\xb,\bar{E})\}_{i\in \{1,\ldots,m-r\}}$, as long as this holds for all $\bar{E}\in \cb{G(\xb)}$ instead of a fixed one. Similarly, Robinson's CQ can be translated as the positive linear independence of the family $\{v_{ii}(\xb,\bar{E})\}_{i\in \{1,\ldots,m-r\}}$ for every $\bar{E}\in \cb{G(\xb)}$~\cite[Prop. 5.1]{weaksparsecq}. This characterization suggested a weak form of nondegeneracy (and Robinson's CQ) that takes into account only a particular subset of $\cb{G(\xb)}$ instead of the whole set, which reads as follows:
\begin{definition}[Def. 3.2 and Def. 5.1 of~\cite{weaksparsecq}]\label{def:wndgwrob}
Let $\xb\in \F$ and let $r$ be the rank of $G(\xb)$. We say that $\xb$ satisfies:
\begin{itemize}
\item {\emph{Weak-nondegeneracy condition} for NSDP} if either $r=m$ or, for each sequence $\seq{x}\to \xb$, there exists some $\bar{E}\in \limsup_{k\in \N} \cb{G(x^k)}$ such that the family $\{v_{ii}(\xb,\bar{E})\}_{i\in \{1,\ldots,m-r\}}$ is linearly independent;
\item {\emph{Weak-Robinson's CQ condition} for NSDP} if either $r=m$ or, for each sequence $\seq{x}\to \xb$, there exists some $\bar{E}\in \limsup_{k\in \N} \cb{G(x^k)}$ such that the family $\{v_{ii}(\xb,\bar{E})\}_{i\in \{1,\ldots,m-r\}}$ is positively linearly independent.
\end{itemize}
\end{definition}
Note that, in general, $\limsup_{k\in \N} \cb{G(x^k)}\subseteq\cb{G(\xb)}$, but the reverse inclusion is not always true, meaning $\cb{G(x)}$ is not necessarily continuous at $\xb$ as a set-valued mapping. It then follows that weak-nondegeneracy is indeed a strictly weaker CQ than nondegeneracy~\cite[Ex. 3.1]{weaksparsecq}. Moreover, in contrast with nondegeneracy, weak-nondegeneracy happens to fully recover LICQ when $G(x)$ is a structurally diagonal matrix constraint in the form $G(x)\doteq\Diag(g_1(x),\ldots,g_m(x))$~\cite[Prop. 3.3]{weaksparsecq}. Similarly, weak-Robinson's CQ is implied by Robinson's CQ and coincides with MFCQ when $G(x)$ is diagonal.

\subsection{Weak constant rank CQs for NSDP}

A straightforward relaxation of weak-nondegeneracy and weak-Robinson's CQ, likewise NLP, leads to our first extension proposal of \nlpcrcq{} and \nlpcpld{} to NSDP: 
\begin{definition}[\sdpwcrcq{}{} and \sdpwcpld{}]\label{sdp:crcq}
	Let $\xb\in \F$ and let $r$ be the rank of $G(\xb)$. We say that $\xb$ satisfies the:
	\begin{itemize}
		\item {\emph{Weak constant rank constraint qualification} for NSDP} (\sdpwcrcq{}) if either $r=m$ or, for each sequence $\seq{x}\to \xb$, there exists some $\bar{E}\in \limsup_{k\in \N} \cb{G(x^k)}$ such that, for every subset $J\subseteq \{1,\ldots,m-r\}$: if the family $\{v_{ii}(\xb,\bar{E})\}_{i\in J}$ is linearly dependent, then $\{v_{ii}(x^k,E^k)\}_{i\in J}$ remains linearly dependent, for all $k\in I$ large enough.
		
		\item {\emph{Weak constant positive linear dependence constraint qualification}} for NSDP (\sdpwcpld{}) if either $r=m$ or, for each sequence $\seq{x}\to \xb$, there exists some $\bar{E}\in \limsup_{k\in \N} \cb{G(x^k)}$ such that, for every \si{subset} $J\subseteq \{1,\ldots,m-r\}$: if the family $\{v_{ii}(\xb,\bar{E})\}_{i\in J}$ is positively linearly dependent, then \jo{the family} $\{v_{ii}(x^k,E^k)\}_{i\in J}$ remains linearly dependent, for all $k\in I$ large enough.
	\end{itemize}

For both definitions, $I\subseteq_{\infty} \N$, and $\{E^k\}_{k\in I}$ is a sequence converging to $\bar E$ and such that  $E^k\in \cb{G(x^k)}$ for every $k\in I$, as required by the Painlev{\'e}-Kuratowski outer limit.

\end{definition}
	
Clearly, weak-nondegeneracy implies \sdpwcrcq{}, which in turn implies \sdpwcpld{}. Also, \si{the condition} weak-Robinson's CQ implies \sdpwcpld{} as well. However, Robinson's CQ and its weak variant are both independent of \sdpwcrcq{}. In fact, the next example shows that \sdpwcrcq{} is not implied by either (weak-)Robinson's CQ or \sdpwcpld{}.

\begin{example}\label{sdp:ex0} Let us consider the constraint 
\[
	G(x)\doteq\begin{bmatrix}
		2x_1+x_2^2 & -x_2^2 \\
		-x_2^2 & 2x_1+x_2^2
	\end{bmatrix}
\]
and note that, for every orthogonal matrix $E$ in the form 
\[
	E\doteq\begin{bmatrix}
a & c\\
b & d
\end{bmatrix},
\]
we have
\[
	v_{11}(x,E)=\begin{bmatrix}
		2\\
		2(a-b)^2 x_2
	\end{bmatrix}
	\quad \mand \quad
	v_{22}(x,E)=\begin{bmatrix}
		2\\
		2(c-d)^2 x_2
	\end{bmatrix}.
\]
Then, at $\bar{x}=0$, we have $v_{11}(\xb,\bar{E})=v_{22}(\xb,\bar{E})=[2,0]^\T$, so they are linearly dependent, but positively linearly independent for all $\bar{E}\in \cb{G(\xb)}$. However, choosing any sequence $\seq{x}\to 0$ such that $x_2^k\neq 0$ for all $k$, it follows that the eigenvalues of $G(x_k)$:
\[
	\lambda_1(G(x^k))=2(x_1+x_2^2) 
	\quad \mand \quad 
	\lambda_2(G(x^k))=2x_1,
\]
are simple, with associated orthonormal eigenvectors 
\[
	u_1(G(x^k))=\left(-\frac{1}{\sqrt{2}},\frac{1}{\sqrt{2}}\right)
	\quad \mand \quad
	u_2(G(x^k))=\left(\frac{1}{\sqrt{2}},\frac{1}{\sqrt{2}}\right),
\]
respectively, for every $k\in \N$. Then, the only sequence $\seq{E}$ such that $E^k\in\cb{G(x^k)}$ for every $k$, up to sign, is given by $a=-1/\sqrt{2}$ and $b=c=d=1/\sqrt{2}$. However, keep in mind that $v_{ii}(x,E)$, $i\in \{1,2\}$, is invariant to the sign of the columns of $E$, so $v_{22}(x^k,E^k)=[2,0]^\T$ and $v_{11}(x^k,E^k)=[2, 4x_2^k]^\T$ are linearly independent for all large $k$. Therefore, we conclude that (weak-)Robinson's CQ holds at $\xb$, and consequently \sdpwcpld{} also holds, but \sdpwcrcq{} does not hold at $\xb$.
\end{example}

Conversely, we show with another counterexample, that \sdpwcrcq{} does not imply (weak-)Robinson's CQ, and neither does \sdpwcpld{}.  

\begin{example}\label{sdp:goodex}
Let us consider the constraint
\[
	G(x)\doteq\begin{bmatrix}
		x & x^2 \\
		x^2 & -x
	\end{bmatrix}
\]
 and the point $\xb=0$. Take any sequence $\seq{x}\to \xb$ such that $x^k\neq 0$ for every $k$, and consider two subsequences of it, indexed by $I_+$ and $I_-$, such that $x^k>0$ for every $k\in I_+$, and $x^k<0$ for every $k\in I_-$. Then, for every $k\in I_+$, we have that:
 \[
	\lambda_1(G(x^k))=x^k \sqrt{(x^k)^2+1} 
	\quad \mand \quad 
	\lambda_2(G(x^k))=-x^k \sqrt{(x^k)^2+1},
\]
are simple, with associated orthonormal eigenvectors uniquely determined (up to sign) by
\[
	u_1(G(x^k))=\frac{1}{\eta^k_1}\left(\frac{1+\sqrt{(x^k)^2+1}}{x^k},1\right)\jo{,}
	\si{\quad \mand} \quad
	u_2(G(x^k))=\frac{1}{\eta^k_2}\left(\frac{1-\sqrt{(x^k)^2+1}}{x^k},1\right),
\]
where
\[
	\eta^k_1\doteq \sqrt{\left(\frac{1+\sqrt{(x^k)^2+1}}{x^k}\right)^2+1} \quad \mand \quad \eta^k_2\doteq \sqrt{\left(\frac{1-\sqrt{(x^k)^2+1}}{x^k}\right)^2+1}.
\] 

Moreover, one can verify that whenever $I_+$ is an infinite set, 
\[
	\lim_{k\in I_+} u_1(G(x^k))= (1,0) \quad \mand \quad \lim_{k\in I_+} u_2(G(x^k))= (0,1).\]
Then, we have that for all $\bar{E}\in\limsup_{k\in I_+} \cb{G(x^k)}$, the vectors
\[
	v_{11}(\xb,\bar{E})=1 \quad \mand \quad v_{22}(\xb,\bar{E})=-1
\]
are positively linearly dependent. And, in addition, since $\eta_1^k\to \infty$ and $\eta_2^k\to 0$, the vectors
\[
	v_{11}(x^k,E^k)=\frac{\eta_1^k+4\sqrt{(x^k)^2+1}-2}{\eta_1^k} \quad \mand \quad v_{22}(x^k,E^k)=\frac{\eta_2^k-4\sqrt{(x^k)^2+1}-2}{\eta_2^k}
\]
are nonzero and have opposite signs; and thus, remain positively linearly dependent, for all large $k\in I_+$.

For the indices $k\in I_-$ the order of $\lambda_1(G(x^k))$ and $\lambda_2(G(x^k))$ is swapped, together with their respective eigenvectors, and we have $\lim_{k\in I_-} u_1(G(x^k))=(0,1)$ and $\lim_{k\in I_-}u_2(G(x^k))=(-1,0)$. Hence, for all $\bar{E}\in\limsup_{k\in I_-} \cb{G(x^k)}$, the vectors
\[
	v_{11}(\xb,\bar{E})=-1 \quad \mand \quad v_{22}(\xb,\bar{E})=1
\]
are also positively linearly dependent. The order of $v_{11}(x^k,E^k)$ and $v_{22}(x^k,E^k)$ is also swapped, so they remain positively linearly dependent for all large $k\in I_-$.

By the above reasoning, observe that any sequence $\seq{x}\to \xb$, such that $x^k\neq 0$ for every $k\in \N$, shows that (weak-)Robinson's CQ fails at $\xb$. Moreover, if $x^k=0$ for infinitely many indices, we may simply take $E^k=\bar{E}=\I_2$ for every $k$, and then $v_{11}(x^k,E^k)=v_{11}(\xb,\bar{E})=1$ and $v_{22}(x^k,E^k)=v_{22}(\xb,\bar{E})=-1$ are positively linearly dependent for every $k\in \N$. This completes checking that \sdpwcpld{} and \sdpwcrcq{} both hold at $\xb$, while (weak-)Robinson's CQ does not.
\end{example}

Just as it happens in NLP, the \sdpwcpld{} condition is strictly weaker than (weak-)Robinson's CQ, and also weaker than \sdpwcrcq{}, which are in turn, independent. Furthermore, let us establish a formal relationship between \sdpwcrcq{} and \sdpwcpld{}, and their NLP counterparts:

\begin{proposition}\label{prop:crcqrecovers}
	Let $G(x)\doteq \Diag(g_1(x)\ldots,g_m(x))$ be a structurally diagonal constraint and let $\xb$ be such that $g_1(\xb)\geq 0,\ldots,g_m(\xb)\geq 0$. Then, the following statements are equivalent:
	\begin{enumerate}
		\item \sdpwcrcq{} holds at $\xb$;
		\item For every $J\subseteq \A(\xb)$, if the set $\{\nabla g_i(\xb)\colon i\in J\}$ is linearly dependent, then $\{\nabla g_i(x)\colon i\in J\}$ is also linearly dependent, for every $x$ close enough to $\xb$;
	\end{enumerate}
	where $\A(\xb)\doteq \{i\in \{1,\ldots,m\}\colon g_i(\xb)=0\}$ is the set of \emph{active indices} at $\xb$.
\end{proposition}
\begin{proof}
Let $r\doteq \rank(G(\xb))$, and note that the result follows trivially if $m=r$. Hence, we will assume that $r<m$. For simplicity, we will also assume that $\mathcal{A}(\xb)=\{1,\ldots,m-r\}$.

\begin{itemize}
	\item \textbf{1 $\Rightarrow$ 2:} By contradiction, suppose that there is some $J\subseteq \A(\xb)$ and a sequence $\seq{x}\to \xb$ such that $\{\nabla g_i(x^k)\colon i\in J\}$ is linearly independent for every $k$, but $\{\nabla g_i(\xb)\colon i\in J\}$ is not. Let $\seq{E}$ and $\bar{E}$ be the sequence and its limit point described in Definition~\ref{sdp:crcq}, for this particular $\seq{x}$. Note that any other set $J'$ that contains $J$ such that $\{\nabla g_{i}(x^k)\colon i\in J'\}$ is linearly independent also fits this description, so let us assume that $J$ is maximal.
	
	 Since $G(x^k)$ is diagonal, every eigenvector $v^k$ associated with an eigenvalue $\lambda^k$ must satisfy $G_{jj}(x^k)v^k_j=\lambda^k v^k_j$ for every $j\in\{1,\ldots,m\}$, which implies $\lambda^k=G_{jj}(x^k)$ or $v^k_j=0$.  
	 Moreover, since $G$ is continuous, the $m-r$ smallest eigenvalues of $G(x^k)$ converge to zero, and consequently, they are bounded from above by \
	 \[
	 	\alpha\doteq \frac{1}{2}\min\{G_{ii}(\bar{x})\colon i\in \{m-r+1,\ldots,m\}\}
	 \] 
for $k$ large enough. 
On the other hand, by continuity of $G$ again, the $r$ largest eigenvalues of $G(x^k)$ are bounded from below by $\alpha$ for all $k$ large enough. Hence, it necessarily holds that $v^k_j=0$ for all $j\in\{m-r+1,...,m\}$ and for all $k$ large enough. That is, $E^k$ has the form
	\begin{equation}\label{sdp:ekform}
		E^k=\begin{bmatrix}
			Q^k\\
			0
		\end{bmatrix}, \ \textnormal{where $Q^k\in \R^{m-r\times m-r}$ is orthogonal,}
	\end{equation}
	for every $k$ large enough. 
	A simple computation shows us that 
	\begin{equation}\label{sdp:diageq0}
		v_{ii}(x^k,E^k)=\sum_{j=1}^{m-r}\nabla g_j(x^k) (Q^k_{ji})^2, \ \mand \ v_{ii}(\xb,\bar{E})=\sum_{j=1}^{m-r}\nabla g_j(\xb) \bar{Q}_{ji}^2
	\end{equation}
	for every $i\in \{1,\ldots,m-r\}$, where $\bar{Q}$ is the submatrix of $\bar{E}$ correspondent to the indices of $Q^k$. Observe that
	\[
		\spn(\{\nabla g_{i}(x^k)\colon i\in J\})=\spn(\{\nabla g_{i}(x^k)\colon i\in \{r+1,\ldots,m\}\}),
	\]
	for all $k$ large enough; otherwise, there would be a subsequence $\{x^k\}_{k\in I}\subseteq \seq{x}$ and another index $j'\not\in J$ such that $\{\nabla g_{i}(x^k)\colon i\in J\cup\{j'\}\}$ is linearly independent for every $k\in I$, contradicting the maximality of $J$. Hence, for every $S\subseteq \{1,\ldots,m-r\}$ we have
	\begin{equation}\label{sdp:diageq1}
		\spn(\{v_{ii}(x^k,E^k)\colon i\in S\})\subseteq \spn(\{\nabla g_i(x^k)\colon i\in J\})
	\end{equation}
	for every large enough $k$. In particular, there exists some $S'\subseteq \{1,\ldots, m-r\}$ with the same cardinality as $J$, such that \eqref{sdp:diageq1} holds with equality for every large $k$. On the other hand, it follows from~\eqref{sdp:diageq0} that
	\[
		\spn(\{v_{ii}(\xb,\bar{E})\colon i\in S'\})\subseteq \spn(\{\nabla g_i(\xb)\colon i\in J\}),
	\]
	and this implies $\spn(\{v_{ii}(\xb,\bar{E})\colon i\in S'\})$ is a linearly dependent set. However, since $\{v_{ii}(x^k,E^k)\colon i\in S'\}$ is linearly independent for all $k$, by \sdpwcrcq{}, we obtain a contradiction.
	
	\item \textbf{2 $\Rightarrow$ 1:} Take $Q^k=\I_{m-r}$ and $E^k$ as in~\eqref{sdp:ekform}, so we have $v_{ii}(x^k,E^k)=\nabla g_{i}(x^k)$ for every $i\in \{1,\ldots,m-r\}$ and every $k\in \N$, and the result follows immediately.
\end{itemize}

\end{proof}
Using analogous arguments to the proposition above, we can also prove the following:
\begin{corollary}\label{cor:cpldrecovers}
Under the same hypotheses of the previous proposition, the following are equivalent:
	\begin{enumerate}
		\item \sdpwcpld{} holds at $\xb$;
		\item For every $J\subseteq \A(\xb)$, if the set $\{\nabla g_i(\xb)\colon i\in J\}$ is positively linearly dependent, then $\{\nabla g_i(x)\colon i\in J\}$ is linearly dependent, for every $x$ close enough to $\xb$.
	\end{enumerate}
\end{corollary}
\begin{proof}
Note, in~\eqref{sdp:diageq0}, that $v_{ii}(x^k,E^k)$ is generated by a nonnegative linear combination of $\nabla g_i(x^k)$, $i\in \{1,\ldots,m-r\}$. Therefore, every argument in the proof of Proposition~\ref{prop:crcqrecovers} can be adapted to prove Corollary~\ref{cor:cpldrecovers}. It suffices to consider positive linear independence, instead of linear independence; and the smallest cone generated by $\{v_{ii}(x^k,E^k)\}_{i\in S}$, instead of the smallest subspace.
\end{proof}

Advancing to the main result of this section, which is to prove that \sdpwcpld{} (and therefore, \sdpwcrcq{}) guarantees the existence of Lagrange multipliers at all local solutions of~\eqref{NSDP}, we get inspiration in the proof of~\cite[Thm. 3.1]{Andreani2005} for NLP, and the proof of~\cite[Thm. 3.2]{weaksparsecq}. That is, we analyse the sequence from Theorem~\ref{thm:minakkt} in terms of the spectral decomposition of its approximate Lagrange multiplier candidates, under \sdpwcpld{}. Then, we use Carathéodory's Lemma~\ref{lem:carath} to construct a bounded sequence from it, that converges to a Lagrange multiplier. As an intermediary step, we also obtain a convergence result of the external penalty method to KKT points under \sdpwcpld{}, a fact that is emphasized in the statement of the next theorem.
\begin{theorem}\label{sdp:mincrcqkkt}
	Let $\{\rho_k\}_{k\in \N}\to \infty$ and $\seq{x}\to\xb\in \F$ be such that
	\[
		\nabla_x L\left(x^k,\rho_k \Pi_{\S^m_+}(-G(x^k))\right)\to 0.
	\]
	If $\xb$ satisfies \sdpwcpld{}, then $\xb$ satisfies the KKT conditions. In particular, every local minimizer of\jo{~}\eqref{NSDP} that satisfies \sdpwcpld{} also satisfies KKT.
\end{theorem}
\begin{proof}
Let $Y^k\doteq \rho_k \Pi_{\S^m_+}(-G(x^k))$, for every $k\in \N$. Recall that we assume $\lambda_1(-G(x^k))\geq \ldots\geq \lambda_m(-G(x^k))$, for every $k$, and denote by $r$ the rank of $\Ker G(\xb)$. Note that when $k$ is large enough, say greater than some $k_0$, we necessarily have $\lambda_i(-G(x^k))=-\lambda_{m-i+1}(G(x^k))< 0$ for all $i\in \{m-r+1,\ldots,m\}$. Let $I\subseteq_{\infty} \N$, and $\{E^k\}_{k\in I}\to \bar{E}$ be such that $E^k\in \cb{G(x^k)}$ for every $k\in I$, as described in Definition~\ref{sdp:crcq}. Then, for each $k\in I$ greater than $k_0$,  the spectral decomposition of $Y^k$ is given by
	\[
		Y^k=\sum_{i=1}^{m-r} \alpha^k_i e_i^k(e_i^k)^\T,
	\]
	where $\alpha_i^k\doteq [\rho_k\lambda_i(-G(x^k))]_+\geq 0$ and $e_i^k$ denotes the $i$-th column of $E^k$, for every $i\in \{1,\ldots,m-r\}$. 
	Since $\nabla_x L(x^k,Y^k)\to 0$, we have
	\begin{equation}\label{sdp:decompstat}
		\nabla f(x^k)-\sum_{i=1}^{m-r} \alpha^k_i DG(x^k)^*\left[e_i^k(e_i^k)^\T\right]\to 0,
	\end{equation}
	but note that
	\[	
		DG(x^k)^*\left[e_i^k (e_i^k)^\top\right]=
		\begin{bmatrix}
			\langle D_{x_1} G(x^k), e_i^k (e_i^k)^\top \rangle\\
			\vdots\\
			\langle D_{x_n} G(x^k), e_i^k (e_i^k)^\top \rangle
		\end{bmatrix}
		=
		\begin{bmatrix}
			(e_i^k)^\top D_{x_1} G(x^k) e_i^k\\
			\vdots\\
			(e_i^k)^\top D_{x_n} G(x^k) e_i^k
		\end{bmatrix}
		=
		v_{ii}(x^k,E^k),
	\]
	so we can rewrite~\eqref{sdp:decompstat} as
	\[
		\nabla f(x^k)-\sum_{i=1}^{m-r} \alpha^k_{i} v_{ii}(x^k,E^k)\to 0.
	\]
	Using Carath{\'e}odory's Lemma~\ref{lem:carath} for the family $\{v_{ii}(x^k,E^k)\}_{i\in \{1,\ldots,m-r\}}$, for each fixed $k\in I$, we obtain some $J^k\subseteq \{1,\ldots,m-r\}$ such that $\{v_{ii}(x^k,E^k)\}_{i\in J^k}$ is linearly independent and
	\begin{equation}\label{sdp:statcarath}
		\nabla f(x^k)-\sum_{i=1}^{m-r} \alpha^k_i v_{ii}(x^k,E^k)=\nabla f(x^k)-\sum_{i\in J} \tilde{\alpha}^k_i v_{ii}(x^k,E^k),
	\end{equation}
	where $\tilde{\alpha}_i^k\geq 0$ for every $k\in I$ and every $i\in {J^k}$. By the infinite pigeonhole principle, we can assume $J^k$ is the same, say equal to $J$, for all $k\in I$ large enough. We claim that the sequences $\{\tilde{\alpha}_i^k\}_{k\in I}$ are all bounded. In order to prove this, suppose that 
	\[
		m^k\doteq \max_{i\in J}\{\tilde\alpha_i^k\}
	\]
	is unbounded with $k\in I$,  
	divide~\eqref{sdp:statcarath} by $m^k$ and note that $m^k\to \infty$ on a subsequence implies that the vectors $v_{ii}(\xb,\bar{E})$, $i\in J$, are positively linearly dependent. On the other hand, the vectors $v_{ii}(x^k,E^k)$, $i\in J$, are linearly independent for all large $k$, which contradicts \sdpwcpld{}. Finally, note that every collection of limit points $\{\bar{\alpha}_i\colon i\in J\}$ of their respective sequences $\seq{\tilde{\alpha}_i}, i\in J$, generates a Lagrange multiplier associated with $\xb$, which is $\bar{Y}\doteq \sum_{i\in J}\bar{\alpha}_i u_i(G(\xb))$. Thus, $\xb$ is a KKT point.
	
	The second part of the statement of the theorem follows from Theorem~\ref{thm:minakkt}.
\end{proof}

Back to Example~\ref{ex:sdpnino}, observe that \sdpwcpld{} does not hold at $\xb=0$, as expected. Indeed, for any sequence $\seq{x}\to 0$ such that $x^k< 0$ for all $k$, the matrix $G(x^k)$ has only simple eigenvalues, for all large $k$, so $E^k\in \cb{G(x^k)}$ is unique up to sign. Without loss of generality, we can assume
	\[
		E^k\doteq \frac{1}{\sqrt{2}}\begin{bmatrix}
		-1 & 1\\
		1 & 1
		\end{bmatrix},
	\]
	and then we have $v_{11}(x^k,E^k)=-2x^k>0$, which is linearly dependent for all $k$ while $v_{11}(\xb,\bar{E})=0$ is positively linearly dependent. Thus Definition~\ref{sdp:crcq} is not satisfied.

\begin{remark}\label{sdp:remnaive}
	In~\cite{crcq-naive}, we presented a different extension proposal of \nlpcrcq{} (and \nlpcpld{}) to NSDP problems with multiple constraints, which is weaker than nondegeneracy (respectively, Robinson's CQ) for a single constraint as in~\eqref{NSDP} only when the zero eigenvalue of $G(\xb)$ is simple. We called this definition the ``naive extension of \nlpcrcq{} (and \nlpcpld{})''. We remark that Definition~\ref{sdp:crcq} coincides with the naive extension of \nlpcrcq{} (and \nlpcpld{}) when zero is a simple eigenvalue of $G(\xb)$, which makes Definition~\ref{sdp:crcq} an improvement of it, or a ``non-naive variant'' of it.
\end{remark}

The phrasing of Theorem~\ref{sdp:mincrcqkkt} was chosen to call the reader's attention to the fact that it is, essentially, a convergence proof of the external penalty method to KKT points, under \sdpwcpld{}. To obtain a more general convergence result, in the next section we introduce new constant rank-type CQs for NSDP that support every algorithm that converges with a more general type of sequential optimality condition. Then, we prove some properties of these new conditions, and we compare them with \sdpwcpld{} and \sdpwcrcq{}.

\section{Stronger sequential-type constant rank CQs for NSDP and global convergence of algorithms}\label{sec:seq}

A more general sequential optimality condition, which was brought from NLP to NSDP by Andreani et al.~\cite{ahv}, is the so-called \emph{Approximate Karush-Kuhn-Tucker} (AKKT) condition. Let us recall one of its many characterizations\footnote{Definition~\ref{nsdp:akkt} coincides with the AKKT condition presented in~\cite[Def. 3.1]{ahv}. See, for instance,~\cite[Prop. 4]{Andreani2020}.}.

\begin{definition}[Def. 4 of~\cite{Andreani2020}]\label{nsdp:akkt}
We say that a point $\xb\in \F$ satisfies the AKKT condition when there exist sequences $\seq{x}\to \xb$ and $\seq{Y}\subseteq\K$, and perturbation sequences $\seq{\delta}\subseteq \R^n$ and $\seq{\Delta}\subseteq \Y$, such that:
\begin{enumerate}
	\item $\nabla_x L(x^k,Y^k) = \delta^k$, for every $k\in \N$;
	\item $G(x^k)+\Delta^k\succeq 0$ and $\langle G(x^k)+\Delta^k, Y^k \rangle=0$, for every $k\in \N$;
	\item $\Delta^k\to 0$ and $\delta^k\to 0$.
\end{enumerate}
\end{definition}

Note that $\seq{Y}$ is a sequence of approximate Lagrange multipliers of $\xb$, in the sense that $Y^k$ is an exact Lagrange multiplier, at $x=x^k$, for the perturbed problem
\begin{equation*}
  \begin{aligned}
    & \underset{x \in \mathbb{R}^{n}}{\text{Minimize}}
    & & f(x) + \langle \xb-x,\delta^k\rangle, \\
    & \text{subject to}
    & & G (x)+\Delta^k \succeq 0.
  \end{aligned}
  \label{NSDP-pert}
\end{equation*}

The main goal in enlarging the class of approximate Lagrange multipliers $Y^k$ and perturbations $\Delta^k$ as in Definition~\ref{nsdp:akkt} instead of considering only the ones given by Theorem~\ref{thm:minakkt}, is to capture the output sequences of a larger class of iterative algorithms. In the next two subsections, we illustrate the previous statement. What is remarkable is that the proof of Theorem~\ref{sdp:mincrcqkkt} can still be somewhat conducted considering this more general class of sequences, arriving at strong global convergence results for such algorithms (Theorem~\ref{sdp:scrcqcq}).

\subsection{Example 1: A safeguarded augmented Lagrangian method}\label{sec:appendix1}

Let us briefly recall a variant of the \textit{Powell-Hestenes-Rockafellar} augmented Lagrangian algorithm that employs a \textit{safeguarding} technique, which is the direct generalization of the one studied in~\cite{bmbook}. The variant we use is also a generalization of~\cite[Pg. 13]{ahss12} and~\cite[Alg. 1]{ahv}, for instance.

  For an arbitrary penalty parameter $\rho>0$ and a \textit{safeguarded multiplier} $\tilde{Y}\succeq 0$, we define $L_{\rho,\tilde{Y}}: \mathbb{R}^{n} \rightarrow \mathbb{R}$ as the \textit{Augmented Lagrangian function} of~\eqref{NSDP}, which is given by
    \begin{equation}\nonumber
    L_{\rho,\tilde{Y}}(x)\doteq f(x)+
    \frac{\rho}{2}
    \left\|\bpcon{-G(x)+\frac{\tilde{Y}}{\rho}}\right\|^{2}-\frac{1}{2\rho}\left\|\tilde{Y}\right\|^{2}.
    \end{equation}
     Since it will be useful in the convergence proof, we compute the gradient of $L_{\rho,\tilde{Y}}$ at $x$ below:
        \begin{equation}\label{eqn:gradlag}
    \nabla L_{\rho,\tilde{Y}}(x)= \nabla f(x)-
    DG(x)^*\left[\rho\bpcon{-G(x)+\si{\frac{\tilde{Y}}{\rho}}\jo{\frac{\tilde{Y}}{\rho}}}\right].
    \end{equation}
    
    Now, we state the algorithm:

\begin{center}
\centering
\begin{minipage}{\linewidth}
\begin{algorithm}[H]
	\caption{Safeguarded augmented Lagrangian method}
	\label{algencan}
	\medskip

    {\it \textbf{Input}:} A sequence $\{\varepsilon_{k}\}_{k\in \N}$ of positive scalars such that $\varepsilon_{k} \rightarrow 0$; a nonempty convex compact set $\mathcal{B} \subset \K$; real parameters $\tau>1$, $\sigma \in (0,1)$, and $\rho_1>0$; and initial points $(x^{0}, \tilde{Y}^{1}) \in \mathbb{R}^{n}\times \mathcal{B}$. Also, define $\|V^{0}\|=\infty$.\\
	\medskip
	
	Initialize $k\leftarrow 1$. Then:
	
	\medskip
	{\it \textbf{Step 1} (Solving the subproblem):} Compute an approximate stationary point $x^k$ of $L_{\rho_k,\tilde{Y}^k}(x)$, that is, a point $x^{k}$ such that 
		\begin{equation*}\label{eqn:approxfeasibility} 
			\|  \nabla L_{\rho_k,\tilde{Y}^k}(x^k) \|\leq \varepsilon_{k};    
		\end{equation*}	
	\medskip
	{\it \textbf{Step 2} (Updating the penalty parameter):}
	Calculate
	\begin{equation}\label{alg:alfeascompmeasure}
		V^{k}\doteq \bpcon{-G(x^k)+\frac{\tilde{Y}^k}{\rho_k}}-\frac{\tilde{Y}^{k}}{\rho_{k}};       
		\end{equation}
	Then,
	\begin{itemize}
	\item[\textbf{a.}] If $k=1$ or $\norm{V^k} \leq \tau \norm{V^{k-1}}$, set
	$\rho_{k+1}\doteq\rho_{k}$; 
	\item[\textbf{b.}] Otherwise, take $\rho_{k+1}$ such that $\rho_{k+1} \geq \gamma\rho_{k}$.
	\end{itemize}
	
	\medskip
	{\it \textbf{Step 3 }(Estimating a new safeguarded multiplier):} Choose any
	$\tilde{Y}^{k+1}\in \mathcal{B}$, set $k \leftarrow k+1$ and go to Step 1.
	
	\medskip
\end{algorithm}
\end{minipage}
\end{center}

\medskip

By the definition of projection we have that $\tilde{Y}^k=\Pi_{\K}(\tilde{Y}^k-\rho_k G(x^k))$ if, and only if, $\tilde{Y}^k,G(x^k)\in \K$ and $\langle\tilde{Y}^k,G(x^k)\rangle=0$, which means that $V^k=0$ if, and only if, the pair $(x^k,\tilde{Y}^k)$ is primal-dual feasible and complementary.  Moreover, note that Algorithm~\ref{algencan} does not necessarily keep a record of the approximate multiplier sequence associated with $\seq{x}$, which is
\begin{equation}\label{alg:approxmult}
	Y^k\doteq \rho_k \bpcon{-G(x^k)+\frac{\tilde{Y}^k}{\rho_k}}.
\end{equation}
Nevertheless, with these multipliers at hand, it is very easy to prove that any feasible limit point $\bar{x}$ of $\seq{x}$ must satisfy AKKT:

\begin{theorem}\label{algenakkt}
Fix any choice of parameters in Algorithm~\ref{algencan} and let $\seq{x}$ be the output sequence generated by it. If $\seq{x}$ has a convergent subsequence $\{x^k\}_{k\in I}\to \bar{x}$, then:
\begin{enumerate}
\item The point $\bar{x}$ is stationary for the problem of minimizing $\frac{1}{2}\|\pcon{-G(x)}\|^2$;
\item If $\xb$ is feasible, then $\xb$ satisfies AKKT.
\end{enumerate}
\end{theorem}
\begin{proof} Let $\{\varepsilon_{k}\}_{k\in \N}\to 0$, $\seq{\tilde{Y}}\subset \mathcal{B}\subset \K$, $\tau>1$, $\sigma \in (0,1)$, and $\rho_1>0$ be the fixed input parameters of Algorithm~\ref{algencan}. Moreover, let $\{\rho_k\}_{k\in \N}$ and $\seq{V}$ computed as in Step 2. For simplicity, let us also assume that $I=\N$.
\begin{enumerate}
\item This part of the proof is standard; see, for instance,~\cite[Prop. 4.3]{Andreani2020};
\item Define $\seq{Y}$ as in~\eqref{alg:approxmult} and take $\Delta^k\doteq V^k$ for all $k\in \N$, where $V^k$ is as given in~\eqref{alg:alfeascompmeasure}. Then, it follows from Step 1 that $\nabla_x L(x^k,Y^k)=\nabla L_{\rho_k,\tilde{Y}^k}(x^k)\to 0$. We also have
\[
	G(x^k)+\Delta^k = \bpcon{G(x^k)-\frac{\tilde{Y}^k}{\rho_k}}
\]
for every $k\in \N$, which yields $\langle Y^k, G(x^k)+\Delta^k\rangle=0$ for every $k$. If $\rho_k\to \infty$, then $V^k\to \pcon{-G(\xb)}$ by definition and $\pcon{-G(\xb)}=0$  because $\xb$ is assumed to be feasible; on the other hand, if $\rho_k$ remains bounded, then $V^k\to 0$ due to Step 2-a. Therefore, $\Delta^k\to 0$ and $\xb$ satisfies AKKT.
\end{enumerate}
\end{proof}
Note that when $\tilde{Y}^k$ is set as zero for every $k$, then Algorithm~\ref{algencan} reduces to the external penalty method, meaning Theorem~\ref{algenakkt} also covers this method.

\subsection{Example 2: A sequential quadratic programming method}\label{sec:appendix2}

Next, we recall Correa and Ram{\'i}rez's~\cite{correaramirez} \textit{sequential quadratic programming} (SQP) method:

\begin{center}
\centering
\begin{minipage}{\linewidth}
\begin{algorithm}[H]
	\caption{General SQP method}
	\label{sqp}
	\medskip

    {\it \textbf{Input}:} 
    A real parameter $\tau>1$, 
    a pair of initial points $(x^{1},Y^1)\in \mathbb{R}^{n}\times \K$, and an approximation of $\nabla^2_x L(x^1,Y^1)$ given by $H^1$.
	\medskip
	
	Initialize $k\leftarrow 1$. Then:
	
	\medskip
	{\it \textbf{Step 1} (Solving the subproblem):} Compute a solution $d^k$, together with its Lagrange multiplier $Y^{k+1}$, of the problem
		\begin{equation}
  \tag{Lin-QP}
  \begin{aligned}
    & \underset{d \in \mathbb{R}^{n}}{\text{Minimize}}
    & & d^\T H^k d + \nabla f(x^k)^\T d, \\
    & \text{subject to}
    & & G(x^k)+DG(x^k)d \in  \K,
  \end{aligned}
  \label{lin-qp}
\end{equation}
and if $d^k=0$, stop;

	\medskip
	{\it \textbf{Step 2} (Step corrections):}
	Perform line search to find a steplength $\alpha^k\in (0,1)$ satisfying \textit{Armijo's rule}
	\begin{equation}
		    f(x^k+\alpha^k d^k)-f(x^k)\leq \tau \alpha^k\nabla f(x^k)^\T d^k.
		\end{equation}

	\medskip
	{\it \textbf{Step 3 }(Approximating the Hessian):} Set $x^{k+1}\leftarrow x^k+\alpha^k d^k$, compute a positive definite approximation $H^{k+1}$ of $\nabla^2_x L(x^{k+1}, Y^{k+1})$, set $k \leftarrow k+1$, and go to Step 1.
	
	\medskip
\end{algorithm}
\end{minipage}
\end{center}

\medskip

The SQP algorithm generates AKKT sequences as well, as it can be seen in the following proposition:

\begin{proposition}\label{prop:sqp}
If there is an infinite subset $I\subseteq_{\infty} \N$ such that $\lim_{k\in I} d^k = 0$ and $\{\|H^k\|\}_{k\in I}$ is bounded, then any limit point $\bar{x}$ of $\{x^k\}_{k\in I}$ satisfies AKKT.
\end{proposition}

\begin{proof}
By the KKT conditions for \eqref{lin-qp}, there exists some $Y^k\succeq 0$ such that
\begin{eqnarray}
	\nabla f(x^k)+H^k d^k - DG(x^k)^*[Y^k]=0\\
	\langle G(x^k)+DG(x^k)d^k, Y^k \rangle=0.
\end{eqnarray}
Set $\Delta^k\doteq DG(x^k)d^k$ for every $k\in I$ and since $d^k\to 0$, we obtain that $\lim_{k\in I} H^k d^k = 0$ and $\lim_{k\in I} \Delta^k=0$. Moreover, since $d^k$ is feasible, $G(x^k)+\Delta^k\succeq  0$. Thus, $\xb$ satisfies AKKT.
\end{proof}

The hypothesis on the convergence of a subsequence of $\seq{d}$ to zero, directly or indirectly, is somewhat common regarding some types of SQP methods, as well as the boundedness of $H^k$ -- see, for instance,~\cite{CPG,correaramirez,qiwei}.

\subsection{Sequential constant rank CQs for NSDP}

Inspired by AKKT, we are led to introduce a small perturbation in \sdpwcpld{} and \sdpwcrcq{}, which makes it stronger, but also brings some useful properties in return. At first, we present it in a form that most resembles Definition~\ref{sdp:crcq}, for comparison purposes. Later, for convenience, we will provide a characterization of it without sequences.

\begin{definition}[\sdpscrcq{} and \sdpscpld{}]\label{sdp:scrcq}
Let $\xb\in \F$ and let $r$ be the rank of $G(\xb)$. We say that $\xb$ satisfies the 
	\begin{enumerate}
		\item  \emph{Sequential CRCQ condition for NSDP} (\sdpscrcq{}) if $r=m$ or, for all sequences $\seq{x}\to \xb$ and $\seq{\Delta}\subseteq \sym$ with $\Delta^k\to 0$, there exists $\{E^k\}_{k\in I}\to \bar{E}$, $I\subseteq_{\infty} \N$, such that $E^k\in \cb{G(x^k)+\Delta^k}$ for every $k\in I$ and, for every subset $J\subseteq \{1,\ldots,m-r\}$: if the family $\{v_{ii}(\xb,\bar{E})\}_{i\in J}$ is linearly dependent, then $\{v_{ii}(x^k,E^k)\}_{i\in J}$ remains linearly dependent, for all $k\in I$ large enough.
		\item \emph{Sequential CPLD condition for NSDP} (\sdpscpld{}) if $r=m$ or, for all sequences $\seq{x}\to \xb$ and $\seq{\Delta}\subseteq \sym$ with $\Delta^k\to 0$, there exists $\{E^k\}_{k\in I}\to \bar{E}$, $I\subseteq_{\infty} \N$, such that $E^k\in \cb{G(x^k)+\Delta^k}$ for every $k\in I$ and, for every subset $J\subseteq \{1,\ldots,m-r\}$: if the family $\{v_{ii}(\xb,\bar{E})\}_{i\in J}$ is positively linearly dependent, then $\{v_{ii}(x^k,E^k)\}_{i\in J}$ remains linearly dependent, for all $k\in I$ large enough.
	\end{enumerate}
\end{definition}

Note that the only difference between Definitions~\ref{sdp:crcq} and~\ref{sdp:scrcq} is the perturbation matrix $\Delta^k\to 0$. In particular, set $\Delta^k\doteq 0$ for every $k$ to see that \sdpscrcq{} and \sdpscpld{} imply \sdpwcrcq{} and \sdpwcpld{}, respectively. 
Moreover, both implications are strict, as we can see in the following example:

\begin{example}\label{sdp:scrcqstrictlystrongercrcq}
Consider the constraint
\[
	G(x)\doteq \begin{bmatrix}
	x & 0 \\ 0 & -x
	\end{bmatrix}
\]
at the point $\xb=0$, so in this case $r=2$. For every sequence $\seq{x}\to \xb$, we have (up to sign)
\[
	\cb{G(x^k)}=\left\{
		\begin{bmatrix}
		1 & 0 \\ 0 & 1
		\end{bmatrix},			
		\begin{bmatrix}
		0 & 1 \\ 1 & 0
		\end{bmatrix}
	\right\},
\] for every $k\in \N$ such that $x^k\neq\xb$, whereas if $x^k=\bar{x}$, then $\cb{G(x^k)}$ is the set of all orthogonal $2\times 2$ matrices. Take $E^k=\I_2$ for every $k\in \N$ to see that both, \sdpwcrcq{} and \sdpwcpld{}, hold at $\xb$, since 
\[
	v_{11}(x^k,E^k)=1 \quad \mand \quad v_{22}(x^k,E^k)=-1
\]
are nonzero and (positively) linearly dependent for every $k\in \N$.

On the other hand, take 
\[
	\Delta^k\doteq \frac{1}{1+(x^k+1)^2}\begin{bmatrix}
		-x^k(x^k-1)^2 & x^k(x^k+1)\\
		x^k(x^k+1) & x^k+2x^k(x^k+1)^2
	\end{bmatrix},
\]
and note that the eigenvectors of $G(x^k)+\Delta^k$ are uniquely determined up to sign. Then, since $v_{ii}(x,E)$, $i\in \{1,2\}$, is invariant to the sign of the columns of $E$, we can assume without loss of generality that any $E^k\in \cb{G(x^k)+\Delta^k}$ has the form
\[
	E^k=
	\frac{1}{\sqrt{1+(x^k+1)^2}}
	\begin{bmatrix}
		-1 & x^k+1\\
		x^k+1 & 1
	\end{bmatrix}		
\]
for every $k\in \N$, if $x^k\neq 0$. Then, for any sequence $\seq{E}$ such that $E^k\in \cb{G(x^k)+\Delta^k}$ for every $k$, we have
\[
	v_{11}(x^k,E^k)=1-(x^k+1)^2 \quad \mand \quad v_{22}(x^k,E^k)=(x^k+1)^2-1,
\]
which are both nonzero whenever $x^k\neq 0$, but if $\bar{E}$ is a limit point of $\seq{E}$, then $v_{11}(\xb,\bar{E})=v_{22}(\xb,\bar{E})=0$. Thus, neither \sdpscrcq{} nor \sdpscpld{} hold at $\xb$.
\end{example}

Furthermore, since nondegeneracy can be characterized as the linear independence of $v_{ii}(\xb,\bar{E})$, $i\in \{1,\ldots,m-r\}$, for every $\bar{E}\in \cb{G(\xb)}$~\cite[Prop. 3.2]{weaksparsecq}, we observe that it implies \sdpscrcq{} (see also Remark~\ref{sdp:strictndgrob} at the end of this section), but this implication is also strict. Let us show this with a counterexample.
\begin{example}
	We analyse the constraint
	\[
		G(x)\doteq\begin{bmatrix}
			x & 0 \\
			0 & x
		\end{bmatrix}
	\]
	at the point $\xb\doteq 0$. For any $x\in \R$ and any arbitrary orthogonal matrix $E\in \R^{2\times 2}$, note that $E$ has the form
	\begin{equation}\label{genortho}
		E= \begin{bmatrix} a & -b \\ b & a\end{bmatrix}, \textnormal{ if } \det(E)=1 \quad \textnormal{or} \quad E= \begin{bmatrix} a & b \\ b & -a\end{bmatrix}, \textnormal{ if } \det(E)=-1
	\end{equation}
	where $a^2+b^2=1$. In both cases, we have 
	\[
		v_{11}(x,E)=v_{22}(x,E)=a^2+b^2=1.\]
	That is, $v_{11}(x,E)$ and $v_{22}(x,E)$ are nonzero and linearly dependent, regardless of $x$ and $E$. Thus, \sdpscrcq{} holds at $\xb$, although nondegeneracy does not. Note that weak-nondegeneracy also fails at $\xb$, in this example.
\end{example}

By Example~\ref{sdp:ex0}, we verify that Robinson's CQ does not imply \sdpscrcq{}; because otherwise, it would also imply \sdpwcrcq{}, contradicting the example.  As for the converse, the counterexample below shows that \sdpscrcq{} does not imply Robinson's CQ either.

\begin{example}\label{sdp:scrcqnotrob}
Consider the constraint
\[
	G(x)\doteq \begin{bmatrix}
		x_1 & x_2 \\
		x_2 & -x_1
	\end{bmatrix}.
\]
Clearly, the only feasible point is $\xb=0$. Then, due to the linearity of $G$, it is immediate to see that Robinson's CQ does not hold at $\xb=0$. On the other hand,
for any $x\in \R^2$ and any orthogonal matrix $E\in \R^{2\times 2}$, note that regardless of the form of $E$ as in~\eqref{genortho}, we have $v_{11}(x,E)\neq 0$, $v_{22}(x,E)\neq 0$, and
\[
	v_{11}(x,E)=-v_{22}(x,E).
\]
Thus, 
\sdpscrcq{} holds at $\xb=0$; see also the characterization of Proposition~\ref{sdp:nonseq}.
\end{example}

Another important consequence of Example~\ref{sdp:scrcqnotrob} is that \sdpscpld{} is strictly weaker than Robinson's CQ.

Next, we will show that \sdpscpld{} (and, consequently, \sdpscrcq{}) is enough to establish equivalence between AKKT and KKT with a small adaptation of the proof of Theorem~\ref{sdp:mincrcqkkt}. Note that in view of Theorem~\ref{thm:minakkt}, any condition that establishes that an AKKT point is also a KKT point is, in particular, a CQ; in addition, such a CQ necessarily supports the global convergence of any algorithm supported by AKKT to KKT points. This includes the algorithms presented in Subsections~\ref{sec:appendix1} and~\ref{sec:appendix2}, and Yamashita, Yabe, and Harada's primal-dual interior point method for NSDP~\cite{yamashitaip} -- for details on the latter, see~\cite{jordanip}. We should also stress that this convergence result neither assumes compactness of the Lagrange multiplier set nor that it is a singleton.
\begin{theorem}\label{sdp:scrcqcq}
	Let $\xb\in \F$ be an AKKT point that satisfies \sdpscpld{}. Then, $\xb$ satisfies KKT.
\end{theorem}
\begin{proof}
	Let $\seq{x}\to \xb$, $\seq{Y}\subseteq \sym_+$, and $\seq{\tilde\Delta}\to 0$ be the AKKT sequences from Definition~\ref{nsdp:akkt}. Since $\lambda_i(G(x^k))>0$ for every $i\in\{1,\ldots,r\}$, where $r$ is the rank of $G(\xb)$, then $\lambda_i(G(x^k)+\tilde\Delta^k)>0$ and $\lambda_{m-i+1}(Y^k)=0$ for every such $i$ and all $k$ large enough. Hence, the spectral decomposition of $Y^k$ can be represented in the format
	\[
		Y^k=\sum_{i=1}^{m-r} \lambda_i(Y^k) u_i^k(u_i^k)^\T
	\]
	where $u_1^k,\ldots,u_{m-r}^k$ are shared orthonormal eigenvectors between $Y^k$ and $G(x^k)+\tilde\Delta^k$, associated with the $m-r$ largest eigenvalues of $Y^k$ and the $m-r$ smallest eigenvalues of $G(x^k)+\tilde\Delta^k$, respectively. Defining $E^k=[u_{1}^k,\ldots,u_{m-r}^k]$ for every $k$, we obtain
	\[
		\nabla_x L(x^k,Y^k)=\nabla f(x^k)-\sum_{i=1}^{m-r}\lambda_{i}(Y^k)v_{ii}(x^k,E^k)\to 0.
	\]
For each $k\in\N$, let $P^k\in \R^{m\times r}$ be a matrix whose columns are orthonormal eigenvectors associated with the $r$ largest eigenvalues of $G(x^k)$, and construct
\jo{
	\begin{equation}\label{sdp:mk}
		M^k\doteq U^k \tilde{M}^k (U^k)^\T,
	\end{equation}
	where
}
\si{\begin{equation}\label{sdp:mk}}
\jo{\[}
	\arraycolsep=4pt\def\arraystretch{1.8}
	\jo{\tilde}M^k\doteq \si{U^k}
	\left[
	\begin{array}{c|c}
		\Diag(\lambda_1(G(x^k)),\ldots,\lambda_r(G(x^k))) & 0\\
		\hline
		0 & \Diag((r+1)\|x^k-\xb\|,\ldots,m\|x^k-\xb\|)
	\end{array}
	\right]
	\si{(U^k)^\T},
\jo{\]}
\si{\end{equation}}
\si{where}\jo{and} $U^k\doteq [P^k,E^k]$ for every $k\in \N$. Note that $M^k\to G(\xb)$ and that the $m-r$ smallest eigenvalues of $M^k$ are simple, if $x^k\neq \bar{x}$, meaning their associated eigenvectors are unique up to sign, when $k$ is large enough. Consequently, $v_{ii}(x^k,E^k)$ is invariant to the choice of $E^k\in \cb{M^k}$, for all such $k$, and every $i\in \{1,\ldots,m-r\}$. The rest of this proof follows the exact same lines as the proof of Theorem~\ref{sdp:mincrcqkkt}.
\end{proof}
\begin{remark}\label{sdp:strictndgrob}
The ``perturbed versions'' of weak-nondegeneracy and 
weak-Robin\-son's CQ, 
in the sense of Definition~\ref{sdp:scrcq}, are nondegeneracy and Robinson's CQ, respectively. In other words, nondegeneracy (respectively, Robinson's CQ) holds at $\xb\in \F$ if, and only if, for every sequence $\seq{x}\to \xb$ and every $\seq{\Delta}\subseteq \sym$ such that $\Delta^k\to 0$\if{ and $G(x^k)+\Delta^k\succeq 0$ for every $k$}\fi, there is some $\bar{E}\in \limsup_{k\in \N} \cb{G(x^k)+\Delta^k}$ such that $\{v_{ii}(\xb,\bar{E})\colon i\in \{1,\ldots,m-r\}\}$ is (positively) linearly independent, where $r=\rank (G(\xb))$. For more details, see~\cite[Rem. 3.1]{weaksparsecq}.
\end{remark}

\section{Relationship with metric subregularity CQ}\label{sec:msr}

Besides convergence of algortihms, the CQs we present also have implications towards stability and error analysis. We make this link by means of establishing a relationship between \sdpscpld{} (and \sdpscrcq{}) and the so-called \textit{metric subregularity CQ} (also known as the \emph{error bound CQ} in NLP), defined in our SDP framework as follows:

\begin{definition}[e.g., Def. 1.1 of~\cite{gfrerermsr}]\label{def:msr}
We say that a feasible point $\xb$ of~\eqref{NSDP} satisfies the \emph{metric subregularity CQ} when there exists some $\gamma>0$ and a neighborhood $\mathcal{V}$ of $\xb$ such that
\[
	\textnormal{dist}(x,\F)\leq \gamma \|\Pi_{\K}(-G(x))\|
\]
for every $x\in \mathcal{V}$. That is, when the set-valued mapping $\mathcal{G}\colon\R^n\rightrightarrows \Y$ that maps $x\mapsto G(x)-\K$ is metric subregular at $(\xb,0)\in \textnormal{graph}(\mathcal{G})$. Here $\textnormal{dist}(x,\F)$ denotes the distance between $x$ and $\F$, and $\textnormal{graph}(\mathcal{G})\subseteq \R\times \Y$ is the graph of $\mathcal{G}$.
\end{definition}

The metric subregularity CQ is implied by Robinson's CQ, which in turn coincides with a similar condition called~\emph{metric regularity CQ}, and it has relevant implications on the stability analysis of optimization problems -- for details, we refer to Ioffe's survey~\cite{ioffe1,ioffe2}. Besides, there are several works addressing the relationship between constant rank constraint qualifications and the metric subregularity CQ in NLP, such as Minchenko and Stakhovski~\cite{rcr}, Andreani et al.~\cite{ahss12}, and others.

We will use a sufficient condition for metric subregularity CQ to hold, originally proposed by Minchenko and Stakhovski~\cite[Thm. 2]{rcr} for NLP problems. We made a simple extension of it to NSDP, which seems not to have been done before in the literature. It is worth mentioning, nevertheless, that the proof we present is essentially the same as the original one, with some minor adaptations to the NSDP context via Moreau's decomposition. 

\begin{proposition}\label{stab:prop1-text}
Let $\xb\in \F$ and assume that $G$ is twice differentiable around $\xb$. For every given $x\in \R^n$, let $\Lambda_{\Pi}(x)$ denote the set of Lagrange multipliers of the problem of minimizing $\|z-x\|$ subject to $G(z)\succeq 0$, $z\in \R^n$. If there exist numbers $\tau>0$ and $\delta>0$ such that $\Lambda_{\Pi}(x)\cap \bar{B}(0,\tau)
\neq \emptyset$ for every $x\in B(\xb,\delta)$, then $\xb$ satisfies metric subregularity CQ. 
\end{proposition}

\begin{proof}
Let $\tau$ and $\delta$ be as described in the hypothesis. Following the proof of~\cite[Thm. 2]{rcr}, note that if $\xb\in \int\F$, then it trivially satisfies metric subregularity CQ, so we will assume that $\xb\in \bd \F$. Let $\delta_0\in (0,\delta)$ be such that
\[
	\frac{4}{\delta_0}\I_n-D^2 G(z)^*[Y]\succeq 0
\]
for all $z\in B(\xb,\delta)$ and all $Y\in \textnormal{cl}(B(0,2\tau))$. Let $x\in B(\xb,\delta_0/2)$ be such that $x\not\in\F$. Although $\Pi_{\F}(x)$ may not be well-defined as a function of $x$, we will use the notation $\Pi_{\F}(x)$ to denote an arbitrary minimizer of $\|z-x\|$ subject to $G(z)\succeq 0$. Then, by definition, we have that $\|\Pi_{\F}(x)-x\|\leq \|\bar{x}-x\|<\delta_0/2$, so $\Pi_{\F}(x)\in B(x,\delta_0/2)$ and, therefore, $\|\Pi_{\F}(x)-\xb\|\leq \|\Pi_{\F}(x)-x\|+\|x-\xb\|< \delta_0$. Let $h\colon \R^n\times \S^m\to \R$ be defined as
\[
	h(z,Y)\doteq \frac{\langle z-x, z-\Pi_{\F}(x) \rangle}{\|x-\Pi_{\F}(x)\|} - \langle G(z), Y\rangle
\]
and note that
\[
	\nabla^2_z h(z,Y)=\frac{2}{\|x-\Pi_{\F}(x)\|}\I_n-D^2 G(z)^*[Y] \ \succeq \ \frac{4}{\delta_0}\I_n- D^2 G(z)^*[Y] \ \succeq \ 0
\]
whenever $z\in B(\xb,\delta)$ and $Y\in \textnormal{cl}(B(0,2\tau))$. Thus, $h(z,Y)$ is convex with respect to its first variable $z\in B(\xb,\delta)$, for every $Y\in \textnormal{cl}(B(0,2\tau))$. Now let us fix an arbitrary $Y\in \Lambda_{\Pi}(x)\cap \textnormal{cl}(B(0,\tau))$, which is nonempty by hypothesis. Recall that, by definition of the set $\Lambda_{\Pi}(x)$, we have that $Y$ is a Lagrange multiplier of the projection problem associated with the point $\Pi_{\F}(x)$. Hence, $2Y$ is a Lagrange multiplier of the problem:
\begin{equation}\label{sdp:auxprobstab}
	\textnormal{Minimize} \ \tilde{f}_x(z) \doteq \|z-x\| + \frac{\langle z-x, z-\Pi_{\F}(x) \rangle}{\|x-\Pi_{\F}(x)\|}, \quad \textnormal{subject to} \ G(z)\succeq 0
\end{equation}
associated with the point $\Pi_{\F}(x)$, which is a local minimizer of $\tilde{f}_x$ since it is elementary to check that $\tilde{f}_x(\Pi_{\F}(x))\geq \|z-x\|$ for every $z\in \F$, by the definition of projection (for details, see~\cite[Lem. 3]{rcr}), with equality at $\Pi_{\F}(x)$. Writing the KKT conditions for the problem~\eqref{sdp:auxprobstab} at $\Pi_{\F}(x)$ with the Lagrange multiplier $2Y\in \textnormal{cl}(B(0,2\tau))$, we obtain
\begin{equation}\label{sdp:auxstatstab}
	\frac{2(\Pi_{\F}(x)-x)}{\|x-\Pi_{\F}(x)\|} - DG(\Pi_{\F}(x))^*[2Y]=0
\end{equation}
with $\langle G(\Pi_{\F}(x)),2Y\rangle=0$, which yields
\begin{equation}\label{sdp:auxineqstab}
	\begin{array}{ll}
	\|x-\Pi_{\F}(x)\| & = -\|x-\Pi_{\F}(x)\| - \langle DG(\Pi_{\F}(x))^*[2Y], x-\Pi_{\F}(x)\rangle\\
	& \leq \langle  G(\Pi_{\F}(x))-G(x), 2Y\rangle\\
	& = -\langle G(x), 2Y\rangle
	\end{array}
\end{equation}
after taking inner products of both sides of~\eqref{sdp:auxstatstab} with $x-\Pi_{\F}(x)$. The middle inequality follows from the definition of adjoint and the convexity of $h(z,Y)$ in the first variable. 
Taking Moreau's decomposition for $G(x)$, we obtain from~\eqref{sdp:auxineqstab} that
\[
	\|x-\Pi_{\F}(x)\| \leq -\langle \Pi_{\K}(G(x)), 2Y \rangle+\langle \Pi_{\K}(-G(x)), 2Y\rangle\leq \langle \Pi_{\K}(-G(x)), 2Y\rangle,
\]
because $Y\succeq 0$, which is self-dual, so $\langle \Pi_{\K}(G(x)), 2Y\rangle\geq 0$; then
\[
	\textnormal{dist}(x,\F) = \|x-\Pi_{\F}(x)\| \leq \|2Y\| \|\Pi_{\K}(-G(x))\|\leq 2\tau \|\Pi_{\K}(-G(x))\|.
\]
Since $x$ was chosen arbitrarily, set $\gamma\doteq 2\tau$ and we are done.
\end{proof}

Now, to compare metric subregularity CQ with \sdpscrcq{} and \sdpscpld{}, we first need to show that they are robust, in the sense they are preserved in a neighborhood of the point of interest. This property may not be clear from Definition~\ref{sdp:scrcq}, but it becomes clear after we exhibit a characterization of it without sequences, as follows:

\begin{proposition}\label{sdp:nonseq}
Let $\xb\in \F$ and let $r$ be the rank of $G(\xb)$.
\begin{itemize}
\item \sdpscrcq{} holds at $\xb$ if, and only if, $r=m$ or, for every $\bar{E}\in \cb{G(\xb)}$, there exists some neighborhood $\mathcal{V}$ of $(\xb,\bar{E})$ such that for all $J\subseteq \{1,\ldots,m-r\}$, we have that if the family $\{v_{ii}(\xb,\bar{E})\}_{i\in J}$ is linearly dependent, then $\{v_{ii}(x,E)\}_{i\in J}$ remains linearly dependent for every $(x,E)\in \mathcal{V}$; 
\item \sdpscpld{} holds at $\xb$ if, and only if, $r=m$ or, for every $\bar{E}\in \cb{G(\xb)}$, there exists some neighborhood $\mathcal{V}$ of $(\xb,\bar{E})$ such that for all $J\subseteq \{1,\ldots,m-r\}$, we have that if the family $\{v_{ii}(\xb,\bar{E})\}_{i\in J}$ is positively linearly dependent, then $\{v_{ii}(x,E)\}_{i\in J}$ remains linearly dependent for every $(x,E)\in \mathcal{V}$.
\end{itemize}
\end{proposition}

\begin{proof}
We will prove only item 1, since item 2 follows analogously. Let $\xb$ satisfy \sdpscrcq{}; by contradiction: suppose that there exists some $\bar{E}\in \cb{G(\xb)}$, some $J\subseteq \{1,\ldots,m-r\}$, and some sequence $\{(x^k,E^k)\}_{k\in \N}\to (\xb,\bar{E})$ such that $\{v_{ii}(\xb,\bar{E})\}_{i\in J}$ is linearly dependent, but $\{v_{ii}(x^k,E^k)\}_{i\in J}$ is linearly independent for every large $k\in \N$. Let $P^k\in \R^{m\times r}$ be a matrix whose columns are orthogonal eigenvectors associated with the $r$ largest eigenvalues of $G(x^k)$, define $U^k\doteq [P^k,E^k]$, and consider $M^k$ as in~\eqref{sdp:mk}. Set $\Delta^k\doteq M^k-G(x^k)$ and note that $v_{ii}(x^k,E^k)$ is invariant to $E^k\in \cb{\Delta^k+G(x^k)}$ when $k$ is large, provided that $x^k\neq \bar{x}$.  This contradicts \sdpscrcq{}.

Conversely, let $\seq{x}\to \bar{x}$ and $\Delta^k\to 0$ be 
any sequences, and let $J\subseteq \{1,\ldots,m-r\}$ be any subset. For each $k$, pick any $E^k\in \cb{G(x^k)+\Delta^k}$ and consider the sequence $\seq{E}$, which is bounded. Let $I\subseteq_{\infty} \N$ and $\bar{E}$ be arbitrary, as long as $\{E^k\}_{k\in I}\to \bar{E}$. Then, by hypothesis, there exists a neighborhood $\mathcal{V}$ of $(\xb,\bar{E})$ such that if $\{v_{ii}(\xb,\bar{E})\}_{i\in J}$ is linearly dependent, then $\{v_{ii}(x^k,E^k)\}_{i\in J}$ is also linearly dependent for all large enough $k\in I$, since $(x^k,E^k)\in \mathcal{V}$ for all such $k$. 
\end{proof}

In light of the equivalence of Proposition~\ref{sdp:nonseq}, we obtain the robustness property.

\begin{proposition}\label{prop:robust}
If \sdpscpld{} holds at $\xb$, then there exists a neighborhood $\mathcal{V}$ of $\xb$ such that \sdpscpld{} also holds for every $x\in \mathcal{V}\cap \F$. Moreover, the same property holds for \sdpscrcq{}.
\end{proposition}

\begin{proof}
%
%
Direct from Proposition~\ref{sdp:nonseq}.
\end{proof}

Now, using Proposition~\ref{prop:robust}, 
it is possible to prove that \sdpscpld{} (and \sdpscrcq{}) implies metric subregularity CQ. We shall do this in the same style as Andreani et al.~\cite{ahss12}:

\begin{theorem}\label{sdp:scpldmetricsub}
If \sdpscpld{} holds at $\xb$ and $G$ is twice differentiable around $\xb$, then $\xb$ satisfies metric subregularity CQ.
\end{theorem}

\begin{proof}
Suppose that metric subregularity CQ does not hold at $\xb$. In view of Proposition~\ref{stab:prop1-text}, there exist sequences $\seq{\tau}\to \infty$ and $\seq{y}\to \xb$ such that $\Lambda(y^k)\cap \textnormal{cl}(B(0,\tau^k))= \emptyset$ for every $k\in \N$. 

Now let $\seq{z}$ be such that $z^k=\Pi_{\F}(y^k)$ for each $k$ and note that $z^k\to \xb$. By the previous proposition, $z^k$ satisfies metric subregularity for all $k$ large enough. Consequently, there exists a sequence  $\seq{Y}\subseteq \sym_+$ such that
\[
	\frac{z^k-y^k}{\|z^k-y^k\|}-DG(z^k)^*[Y^k]=0
\]
and $\langle G(z^k),Y^k\rangle=0$ for every $k$, which implies that $\lambda_i(Y^k)=0$ for every $i\in \{m-r+1,\ldots,m\}$ and every $k\in \N$. Let $U^k$ be an arbitrary matrix that diagonalizes $Y^k$ and let $E^k$ be the part of it that corresponds to the $m-r$ smallest eigenvalues of $G(z^k)$. So
\begin{equation}\label{sdp:eqms}
	\frac{z^k-y^k}{\|z^k-y^k\|}-\sum_{i=1}^{m-r} \lambda_i(Y^k) v_{ii}(x^k,E^k)=0.
\end{equation}
Again, by Caratheodory' lemma (cf. Lemma \ref{lem:carath}) and the infinite pigeonhole principle, we obtain a set $J\subseteq \{1,\ldots,m-r\}$ such that $\{v_{ii}(x^k,E^k)\colon i\in J\}$ is linearly independent and $\sum_{i=1}^{m-r}\lambda_i(Y^k)v_{ii}(x^k,E^k)=\sum_{i\in J}\alpha_i^k v_{ii}(x^k,E^k)$ for every $k$ where $\alpha_i^k \lambda_i(Y^k) > 0$ for all $i\in J$. Then, recall from the definition that $Y^k\in \Lambda(y^k)$, so $\|Y^k\|> \tau^k\to \infty$. Let $m^k\doteq \max\{\alpha_i^k\colon i\in J\}$ and divide~\eqref{sdp:eqms} by $m^k$ to obtain that $\{v_{ii}(\xb,\bar{E})\colon i\in J\}$ is linearly dependent for every limit point $\bar{E}$ of $\seq{E}$, which contradicts \sdpscpld{} at $\xb$.
\end{proof}

%
%
%
%
%
%
%
%
%
%
%
%
%
%

\section{Conclusion}\label{sec:conclusion}

There are few constraint qualifications available for NSDP, and as far as we know, the use of CQs in the global convergence of algorithms is somewhat limited to nondegeneracy and Robinson's CQ. In contrast, several constraint qualifications have been defined for NLP over the past decades, mostly improving the global convergence of algorithms beyond the case when the set of Lagrange multipliers is bounded. We are in a path to extend these CQs to conic contexts, such as NSDP, that started in~\cite{crcq-naive}. In fact, the results of this paper can be considered a significant improvement of~\cite{crcq-naive} based on our previous developments in~\cite{weaksparsecq}. We introduced two weak constant rank CQs for NSDP, called \sdpwcrcq{} and \sdpwcpld{}, which are essentially ``diagonal extensions'' of their NLP counterparts, in the sense of Proposition~\ref{prop:crcqrecovers}. Namely, one can embed an NLP problem using a structurally diagonal semidefinite constraint and both conditions are preserved. This is a fairly unusual property as this approach usually induces a degenerate NSDP problem; we however believe that this, in some sense, provides a sound mathematical consistency to our approach. These conditions were used to prove convergence of an external penalty method to stationary points, but any application beyond that, besides the mere existence of Lagrange multipliers, is still a subject for investigation. However, they were the starting points for introducing stronger constant rank CQs, called \sdpscrcq{} and \sdpscpld{}, with more interesting properties, such as the convergence theory of a larger class of algorithms such as augmented Lagrangians, sequential quadratic programming, and interior point methods, and a property related with the ability to compute error bounds under these conditions. We believe that several other applications of constant rank CQs will appear in the literature, such as the computation of the derivative of the value function of a parameterized NSDP problem and the computation of second-order necessary optimality conditions. In NLP, constant rank CQs are used to define a strong second-order necessary optimality condition that depends on a single Lagrange multiplier, rather than on the full set of Lagrange multipliers, which we believe will be the case for conic problems as well. It is also the case that constant rank conditions provide the adequate assumptions for guaranteeing global convergence of algorithms to second-order stationary points, which has not been considered yet in the conic programming literature. 

This paper leaves several interesting open questions that can be addressed in future works, such as the use of \sdpwcrcq{} and \sdpwcpld{} in algorithms other than external penalty methods, and the analysis of some stability properties under the conditions introduced in this manuscript. It is also worth recalling that although our conditions were defined by means of sequences, which seems appropriate when talking about convergence of algorithms, we also provided characterizations of them without sequences, in a more classical way, which should foster new applications. 

The relationship among the CQs we presented in this paper, and existing ones, is summarized in the following diagram, where (solid) arrows represent (strict) implications, existing CQs are in blue boxes, and new CQs are in green boxes.

\tikzstyle{old} = [rectangle, rounded corners, minimum width=1cm, minimum height=0.5cm,text centered, draw=black, fill=blue!20]
\tikzstyle{new} = [rectangle, rounded corners, minimum width=1cm, minimum height=0.5cm,text centered, draw=black, fill=green!20]
\tikzstyle{weak} = [rectangle, rounded corners, minimum width=1cm, minimum height=0.5cm,text centered, draw=black, fill=gray!20]
\tikzstyle{naive} = [rectangle, rounded corners, minimum width=1cm, minimum height=0.5cm,text centered, draw=black, fill=purple!20]
\tikzstyle{arrow} = [thick,->,>=stealth]
\tikzstyle{arrowdash} = [thick,->,dashed,>=stealth]

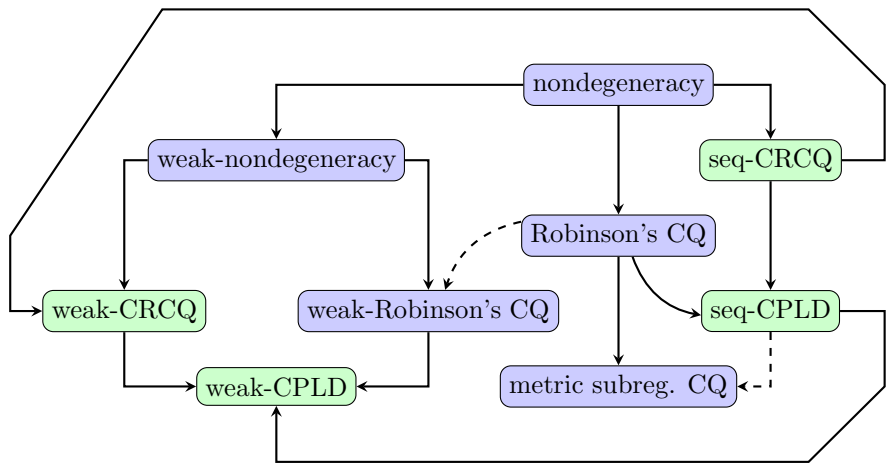
\begin{figure}[htb]
		\centering
        \begin{tikzpicture}[node distance=1cm]
        	\node (ndg) [old] {nondegeneracy};
        	\node (wndg) [old, left of=ndg,xshift=-3.5cm,yshift=-1cm] {weak-nondegeneracy};
            \node (wrob) [old, right of=wndg,xshift=1cm,yshift=-2cm] {weak-Robinson's CQ};
            \node (wcrcq) [new, left of=wndg,xshift=-1cm,yshift=-2cm] {\sdpwcrcq{}};
            \node (wcpld) [new, below of=wndg,yshift=-2cm] {\sdpwcpld{}};
			\node (rob)[old, below of=ndg, yshift=-1cm] {Robinson's CQ};  
            \node (scrcq) [new, right of=ndg,xshift=1cm,yshift=-1cm] {\sdpscrcq{}};
            \node (scpld) [new, below of=scrcq,yshift=-1cm] {\sdpscpld{}};
            \node (msr) [old, below of=rob,yshift=-1cm] {metric subreg. CQ};
            \draw [arrow] (ndg) -- (rob);
            \draw [arrow] (rob) -- (msr);
			\draw [arrow] (ndg) -| (scrcq);
			\draw [arrow] (scrcq) -- (scpld);
			\draw [arrow] (scrcq) -| (3.5,0) -- (2.5,1) -- (-6,1) -- (-8,-2) -- (-8,-3) -- (wcrcq);
			\draw [arrow] (scpld) -| (3.5,-4) -- (2.5,-5) -- (-4.5,-5) -- (wcpld);
			\draw [arrowdash] (scpld) |- (msr);
			\draw [arrow] (rob) to [bend right] (scpld);
			\draw [arrow] (ndg) -| (wndg);
			\draw [arrow] (wndg) -| (wcrcq);
			\draw [arrow] (wndg) -| (wrob);
			\draw [arrow] (wrob) |- (wcpld);
			\draw [arrow] (wcrcq) |- (wcpld);
			\draw [arrowdash] (rob) to [bend right] (wrob);
        \end{tikzpicture}
    \caption{Relationship among the new constraint qualifications and some of the existing ones.}
    \label{fig:sequentialkkt}
\end{figure}

\jo{\newpage}
\bibliographystyle{siamplain}

\begin{thebibliography}{10}

\bibitem{aes2}
{\sc R.~Andreani, C.~E. Echag\"ue, and M.~L. Schuverdt}, {\em Constant-rank
  condition and second-order constraint qualification}, Journal of Optimization
  theory and Applications, 146 (2010), pp.~255--266,
  \url{https://doi.org/10.1007/s10957-010-9671-8}.

\bibitem{ZZerratum}
{\sc R.~Andreani, E.~H. Fukuda, G.~Haeser, H.~Ram{\'i}rez~C., D.~O. Santos,
  P.~J.~S. Silva, and T.~P. Silveira}, {\em {Erratum to: New constraint
  qualifications and optimality conditions for second order cone programs}}, To
  appear in Set-Valued and Variational Analysis,  (2021),
  \url{https://doi.org/10.1007/s11228-021-00573-5}.

\bibitem{jordanip}
{\sc R.~Andreani, E.~H. Fukuda, G.~Haeser, D.~O. Santos, and L.~D. Secchin},
  {\em On the use of {J}ordan algebras for improving global convergence of an
  augmented {L}agrangian method in nonlinear semidefinite programming}, To
  appear in Computational Optimization and Applications,  (2021),
  \url{https://doi.org/10.1007/s10589-021-00281-8}.

\bibitem{Andreani2020}
{\sc R.~Andreani, W.~G{\'o}mez, G.~Haeser, L.~M. Mito, and A.~Ramos}, {\em {On
  optimality conditions for nonlinear conic programming}}, tech. report, 2020,
  \url{http://www.optimization-online.org/DB_HTML/2020/03/7660.html} (accessed
  2020/05/18).

\bibitem{ahm10}
{\sc R.~Andreani, G.~Haeser, and J.~M. Mart{\'i}nez}, {\em On sequential
  optimality conditions for smooth constrained optimization}, Optimization, 60
  (2011), pp.~627--641, \url{http://dx.doi.org/10.1080/02331930903578700}.

\bibitem{weaksparsecq}
{\sc R.~Andreani, G.~Haeser, L.~M. Mito, and H.~Ram{\'i}rez}, {\em Weak notions
  of nondegeneracy in nonlinear semidefinite programming}, tech. report, 2020,
  \url{https://arxiv.org/abs/2012.14810v1} (accessed 2021/02/16).

\bibitem{crcq-naive}
{\sc R.~Andreani, G.~Haeser, L.~M. Mito, H.~Ram{\'i}rez, D.~O. Santos, and
  T.~P. Silveira}, {\em Naive constant rank-type constraint qualifications for
  multifold second-order cone programming and semidefinite programming}, To
  appear in Optimization Letters,  (2021),
  \url{https://doi.org/10.1007/s11590-021-01737-w} (accessed 2021/02/16).

\bibitem{ahss12}
{\sc R.~Andreani, G.~Haeser, M.~L. Schuverdt, and P.~J.~S. Silva}, {\em A
  relaxed constant positive linear dependence constraint qualification and
  applications}, Mathematical Programming, Series A, 135 (2012), pp.~255--273,
  \url{https://doi.org/10.1007/s10107-011-0456-0}.

\bibitem{CPG}
{\sc R.~Andreani, G.~Haeser, M.~L. Schuverdt, and P.~J.~S. Silva}, {\em Two new
  weak constraint qualifications and applications}, SIAM Journal on
  Optimization, 22 (2012), pp.~1109--1135,
  \url{http://dx.doi.org/10.1137/110843939}.

\bibitem{ahv}
{\sc R.~Andreani, G.~Haeser, and D.~S. Viana}, {\em Optimality conditions and
  global convergence for nonlinear semidefinite programming}, Mathematical
  Programming, Series A, 180 (2020), pp.~203--235,
  \url{http://dx.doi.org/10.1007/s10107-018-1354-5}.

\bibitem{amrs2}
{\sc R.~Andreani, J.~Mart\'{\i}nez, A.~Ramos, and P.~J.~S. Silva}, {\em Strict
  constraint qualifications and sequential optimality conditions for
  constrained optimization}, Mathematics of Operations Research, 43 (2018),
  pp.~693--717, \url{https://doi.org/10.1287/moor.2017.0879}.

\bibitem{Andreani2005}
{\sc R.~Andreani, J.~M. Mart{\'i}nez, and M.~L. Schuverdt}, {\em On the
  relation between constant positive linear dependence condition and
  quasinormality constraint qualification}, Journal of Optimization Theory and
  Applications, 125 (2005), pp.~473--485,
  \url{https://doi.org/10.1007/s10957-004-1861-9}.

\bibitem{bertsekasnl}
{\sc D.~P. Bertsekas}, {\em Nonlinear {P}rogramming}, Athenas Scientific.
  Belmont, Mass, 1999.

\bibitem{bmbook}
{\sc E.~Birgin and J.~M. Mart\'inez}, {\em Practical Augmented {L}agrangian
  Methods for Constrained Optimization}, SIAM Publications. Philadelphia, 2014.

\bibitem{bshapiro}
{\sc J.~F. Bonnans and A.~Shapiro}, {\em Pertubation Analysis of Optimization
  Problems}, Springer-Verlag. Berlin, 2000.

\bibitem{boy04}
{\sc S.~Boyd and L.~Vandenberghe}, {\em Convex Optimization}, Cambridge
  University Press, 2004.

\bibitem{correaramirez}
{\sc R.~Correa and H.~Ram{\'i}rez~C.}, {\em A global algorithm for nonlinear
  semidefinite programming}, SIAM Journal on Optimization, 15 (2004),
  pp.~303--318, \url{https://doi.org/10.1137/S1052623402417298}.

\bibitem{gfrerermsr}
{\sc H.~Gfrerer}, {\em First order and second order characterizations of metric
  subregularity and calmness of constraint set mappings}, SIAM Journal on
  Optimization, 21 (2011), pp.~1439--1474,
  \url{https://doi.org/10.1137/100813415}.

\bibitem{haeserip}
{\sc G.~Haeser}, {\em On the global convergence of interior-point nonlinear
  programming algorithms}, Computational and Applied Mathematics, 29 (2010),
  pp.~125--138, \url{https://doi.org/10.1590/S1807-03022010000200003}.

\bibitem{ioffe1}
{\sc A.~D. Ioffe}, {\em Metric regularity - {A} survey. {P}art {I}. {T}heory},
  Journal of the Australian Mathematical Society, 101 (2016), pp.~188--243,
  \url{https://doi.org/10.1017/S1446788715000701}.

\bibitem{ioffe2}
{\sc A.~D. Ioffe}, {\em Metric regularity - {A} survey. {P}art {II}.
  {A}pplications}, Journal of the Australian Mathematical Society, 101 (2016),
  pp.~376--417, \url{https://doi.org/10.1017/S1446788715000695}.

\bibitem{janin}
{\sc R.~Janin}, {\em Directional derivative of the marginal function in
  nonlinear programming}, Mathematical Programming Studies, 21 (1984),
  pp.~127--138, \url{https://doi.org/10.1007/BFb0121214}.

\bibitem{rcr}
{\sc L.~Minchenko and S.~Stakhovski}, {\em On relaxed constant rank regularity
  condition in mathematical programming}, Optimization, 60 (2011),
  pp.~429--440, \url{https://doi.org/10.1080/02331930902971377}.

\bibitem{param}
{\sc L.~Minchenko and S.~Stakhovski}, {\em Parametric nonlinear programming
  problems under the relaxed constant rank condition}, SIAM Journal on
  Optimization, 21 (2011), pp.~314--332,
  \url{https://doi.org/10.1137/090761318}.

\bibitem{Moreau}
{\sc J.~J. Moreau}, {\em D{\'e}composition orthogonale d'un espace hilbertien
  selon deux cones mutuellement polaires}, Comptes Rendus de l'Academie des
  Sciences de Paris, 255 (1962), pp.~238--240.

\bibitem{qiwei}
{\sc L.~Qi and Z.~Wei}, {\em On the constant positive linear dependence
  conditions and its application to {SQP} methods}, SIAM Journal on
  Optimization, 10 (2000), pp.~963--981,
  \url{https://doi.org/10.1137/S1052623497326629}.

\bibitem{Robinson1976}
{\sc S.~M. Robinson}, {\em First-order conditions for general nonlinear
  optimization}, SIAM Journal on Applied Mathematics, 30 (1976), pp.~597--610,
  \url{https://doi.org/10.1137/0130053}.

\bibitem{rwets}
{\sc R.~T. Rockafellar and R.~Wets}, {\em Variational Analysis}, Grundlehren
  der mathematischen Wissenschaften, v. 317. Springer-Verlag Berlin Heidelberg.
  Berlin, 2009.

\bibitem{Shapiro1997}
{\sc A.~Shapiro}, {\em {First and second order analysis of nonlinear
  semidefinite programs}}, Mathematical Programming, Series B, 77 (1997),
  pp.~301--320, \url{https://doi.org/10.1007/BF02614439}.

\bibitem{shapfan}
{\sc A.~Shapiro and M.~K.~H. Fan}, {\em On eigenvalue optimization}, SIAM
  Journal on Optimization, 5 (1995), pp.~552--569,
  \url{https://doi.org/10.1137/0805028}.

\bibitem{yamashitaip}
{\sc H.~Yamashita, H.~Yabe, and K.~Harada}, {\em A primal-dual interior point
  method for nonlinear semidefinite programming}, Mathematical Programming,
  Series A, 135 (2012), pp.~89--121,
  \url{https://doi.org/10.1007/s10107-011-0449-z}.

\bibitem{ZZ}
{\sc Y.~Zhang and L.~Zhang}, {\em New constraint qualifications and optimality
  conditions for second order cone programs}, Set-Valued and Variational
  Analysis, 27 (2019), pp.~693--712,
  \url{https://doi.org/10.1007/s11228-018-0487-2}.
\end{thebibliography}

\end{document}